\documentclass[11pt,a4paper]{article}
\usepackage{amsthm, amssymb, amsmath, a4wide}
\usepackage[all]{xy} 
\usepackage{paralist} 
\usepackage{graphicx} 
\usepackage[sans]{dsfont} 
\usepackage{varioref} 

\author{Pinaki Mondal}
\title{General Bezout-type theorems}

\makeatletter

\newcommand{\Rmnum}[1]{\expandafter\@slowromancap\romannumeral #1@}
\makeatother

\let\oldref\ref
\newcommand{\mathref}[1]{$($\oldref{#1}$)$}

\DeclareMathOperator\Char{char}

\DeclareMathOperator\Div{div} 
 
\DeclareMathOperator\gr{gr}
 
\DeclareMathOperator\ord{ord} 
 
\DeclareMathOperator\proj{Proj}

\DeclareMathOperator\spec{Spec}
\DeclareMathOperator\supp{Supp}
\DeclareMathOperator\vol{Vol}

\newcommand{\scrF}{\ensuremath{\mathcal{F}}}
\newcommand{\scrG}{\ensuremath{\mathcal{G}}}

\newcommand{\scrI}{\ensuremath{\mathcal{I}}}
\newcommand{\scrJ}{\ensuremath{\mathcal{J}}}

\newcommand{\scrM}{\ensuremath{\mathcal{M}}}

\newcommand{\scrP}{\ensuremath{\mathcal{P}}}
\newcommand{\scrQ}{\ensuremath{\mathcal{Q}}}

\newcommand{\psif}{\ensuremath{\psi_\scrF}}

\newcommand{\cc}{\ensuremath{\mathbb{C}}}
\newcommand{\kk}{\ensuremath{\mathbb{K}}}
\newcommand{\nn}{\ensuremath{\mathbb{N}}}
\newcommand{\pp}{\ensuremath{\mathbb{P}}}

\newcommand{\rr}{\ensuremath{\mathbb{R}}}
\newcommand{\zz}{\ensuremath{\mathbb{Z}}}

\newcommand{\affine}[2]{\ensuremath{\mathbb{A}^{#1}(#2)}}

\newcommand{\ank}{\affine{n}{\kk}}

\newcommand{\aaa}{\ensuremath{\mathfrak{a}}}

\newcommand{\ppp}{\ensuremath{\mathfrak{p}}}
\newcommand{\qqq}{\ensuremath{\mathfrak{q}}}
\newcommand{\rrr}{\ensuremath{\mathfrak{r}}}

\newcommand{\sheaf}{\ensuremath{\mathcal{O}}}

\newcommand{\im}{\ensuremath{\Rightarrow}}

\newcommand{\dsum}{\ensuremath{\bigoplus}}
\newcommand{\into}{\ensuremath{\hookrightarrow}}
\newcommand{\onto}{\twoheadrightarrow}

\newcommand{\finv}{\ensuremath{f^{-1}}}

\newtheorem{thm}{Theorem}[section]
\newtheorem*{thm*}{Theorem}
\newtheorem{lemma}[thm]{Lemma}
\newtheorem*{lemma*}{Lemma}
\newtheorem{lemma-in-thm}{Lemma}[thm]

\newtheorem{prop}[thm]{Proposition}
\newtheorem*{prop*}{Proposition}
\newtheorem{cor}[thm]{Corollary}
\newtheorem{claim}{Claim}[thm]
\newtheorem*{claim*}{Claim}
\newtheorem{prolemma}[claim]{Lemma}
\newtheorem*{conjecture*}{Conjecture}

\theoremstyle{definition} 
\newtheorem{example}[thm]{Example}
\newtheorem*{example*}{Example}

\newtheorem*{defn*}{Definition}

\newtheorem*{definotation*}{Definition-Notation}
\newtheorem*{fact*}{Fact}

\newtheorem*{facts*}{Facts}

\newtheorem{rem}[thm]{Remark}

\newtheorem*{reminition*}{Remark-Definition}

\newtheorem*{remtation*}{Remark-Notation}
\newtheorem{bold-question}[thm]{Question}

\newtheorem*{bold-note*}{Note}

\theoremstyle{remark}
\newtheorem*{rem*}{Remark}
\newtheorem*{note*}{Note}
\newtheorem*{notation*}{Notation}
\newtheorem*{question*}{Question}

\newtheorem*{questions*}{Questions}

\theoremstyle{plain}

\newcounter{Cases}

\newcounter{UnorderedProofTempCtr}
\newcommand{\tempcommand}{}

\newcommand{\WP}{\ensuremath{{\bf W}\pp}}

\newcommand{\complete}{\completee\ }
\newcommand{\completee}{complete}

\newcommand{\filtrationchar}{\ensuremath{\mathcal{F}}}
\newcommand{\filtrationring}{\ensuremath{A}}
\newcommand{\profing}{\profingg{\filtrationring}{\filtrationchar}}
\newcommand{\profingg}[2]{\ensuremath{{#1}^{#2}}}
\newcommand{\profinggg}[1]{\profingg{\filtrationring}{#1}}

\newcommand{\gring}{\gringg{\filtrationring}{\filtrationchar}}
\newcommand{\gringg}[2]{\ensuremath{\gr {#1}^{#2}}}

\newcommand{\ld}{\mathfrak{L}}

\begin{document}

\renewcommand{\affine}[2]{\ensuremath{{#2}^{#1}}}
\newcommand{\identity}{\mathds{1}}
\renewcommand{\thefootnote}{\fnsymbol{footnote}} 

\newcommand{\anc}{\affine{n}{\cc}}
\newcommand{\af}{\profingg{A}{\scrF}}
\newcommand{\xf}{X^\scrF}
\newcommand{\kxf}{\kk[X]^\scrF}
\renewcommand{\filtrationring}{\ensuremath{A}}
\renewcommand{\filtrationchar}{\ensuremath{\mathcal{F}}}
\newcommand{\adelta}{\profingg{A}{\delta}}
\newcommand{\xdelta}{X^\delta}
\newcommand{\locali}{\sheaf_{V_{i}, \xdelta}}
\newcommand{\localj}{\sheaf_{V_{j}, \xdelta}}

\maketitle

\begin{abstract}
In this sequel to \cite{sub1} we develop Bezout type theorems for semidegrees (including an explicit formula for {\em iterated semidegrees}) and an inequality for subdegrees. In addition we prove (in case of surfaces) a Bernstein type theorem for the number of solutions of two polynomials in terms of the mixed volume of planar convex polygons associated to them (via the theory of Kaveh-Khovanskii \cite{khovanskii-kaveh} and Lazarsfeld-Mustata \cite{lazarsfeld-mustata})
\end{abstract}

\tableofcontents

\section{Introduction} \label{sec-intro}
{\bf Disclaimer:} This is an unpolished draft of the article. A clearer exposition (with more complete reference) is in order and this submission will be updated in a few days.\\

This article is a sequel to \cite{sub1}. In it we develop affine Bezout type theorems. In Section \ref{sec-bezout} we show that given a polynomial system of $n$ equations on an $n$-dimensional affine variety, there are subdegrees which add nothing at infinity to generic fibers. In Section \ref{sec-Bezout} we find estimates for the number of solutions of the system in terms of the properties of the subdegree. The estimate is exact if the subdegree turns out to be a semidegree. If in addtion the semidegree is of a special class called {\em iterated} semidegrees, then the formula turns out to be explicit and this is handled in Section \ref{sec-iterated}. Finally, in Section \ref{sec-dim-2}, we show that the estimate for subdegrees is exact if $n = 2$. We also give an interpretation of this estimate in terms of the mixed volume of planar convex polygons associated to the subdegrees (via the theory of Kaveh-Khovanskii \cite{khovanskii-kaveh} and Lazarsfeld-Mustata \cite{lazarsfeld-mustata}).

\section{Existence of Intersection Preserving Filtrations} \label{sec-filtrexistence}
\newcommand{\ag}{\profingg{A}{\scrG}}
\newcommand{\xg}{X^\scrG}
\renewcommand{\filtrationring}{\kk[X]}
\newcommand{\apk}{\affine{p}{\kk}}
\newcommand{\aqk}{\affine{q}{\kk}}
\newcommand{\ark}{\affine{r}{\kk}}
\newcommand{\littleoh}{o}

Let $X$ be an affine variety over $\kk$. Recall that for subsets $V_1, \ldots, V_m$ of $X$, a completion $\psi: X \into Z$ is said to {\em preserve the intersection of $V_1, \ldots, V_m$ at $\infty$} if $\overline{V}_1 \cap \cdots \cap \overline{V}_m \cap X_\infty =  \emptyset$, where $X_\infty := Z\setminus X$ is the set of `points at infinity' and $\overline{V}_j$ is the closure of $V_j$ in $Z$ for every $j$.

\begin{lemma} \label{preservation-criterion}
Let $\scrF = \{F_d: d \geq 0\}$ be a complete filtration on $A := \kk[X]$, and $\psi_\scrF : X \into \xf := \proj \af$ be the corresponding completion. 
\begin{compactenum}
\item \label{homogeneous-closure} For each ideal $\qqq$ of $A$, let $\qqq^\scrF := \dsum_{d \geq 0}(\qqq \cap F_d) \subseteq \af$. Then the closure of $V(\qqq) \subseteq X$ in $\xf$ is $V(\qqq^\scrF)$. 
\item Let $V_1, \ldots, V_m$ be Zariski closed subsets of $X$ with $V_i = V(\qqq_i)$ for ideals $\qqq_i \subseteq A$ for each $i$. Let $\scrI$ be the ideal of $\af$ generated by $\qqq^\scrF_1, \ldots, \qqq^\scrF_m$ and $(1)_1$. Then $\psi_\scrF$ preserves the intersection of $V_1, \ldots, V_m$ at $\infty$ iff the $\sqrt \scrI \supseteq \af_+$, where $\af_+ := \dsum_{d>0}F_d$ is the {\em irrelevant} ideal of $\af$.
\end{compactenum}
\end{lemma}

\begin{proof}
\begin{asparaenum}
\item Recall (example \ref{proj-example4}) that there exists $d > 0$ such that $(\af)^{[d]} := \dsum_{k \geq 0}F_{kd}$ is generated by $F_d$ as a $\kk$-algebra. Define a new filtration $\scrG := \{G_k: k \geq 0\}$ on $A$ by $G_k := F_{kd}$. Let $\{1,g_1, \ldots, g_m\}$ be a $\kk$-vector space basis of $G_1$. Then $\ag \cong (\af)^{[d]}$ and by corollary \ref{completecor}, $\xg := \proj \ag$ is the closure in $\pp^m(\kk)$ of $\phi(X)$, where $\phi: X \to \affine{m}{\kk}$ is defined by: $\phi(x) := (g_1(x), \ldots, g_m(x))$.

\hspace{\itemindent} Let $\qqq$ be an ideal of $A$ and $V := V(\qqq)$ be the Zariski closed subset of $X$ defined by $\qqq$. Let $\ppp := \ker \phi^*$ and $\rrr := (\phi^*)^{-1}(\qqq)$, where $\phi^*: \kk[y_1, \ldots, y_m] \to A$ is the pull back by means of $\phi$. Identify $X$ with $V(\ppp)$ and $V$ with $V(\rrr)$ in $\affine{m}{\kk}$. Then $\xg$ and the closure ${\overline V}^\scrG$ of $V$ in $\xg$ are the Zariski closed subsets of $\pp^m(\kk)$ determined by the {\em homogenizations} $\tilde \ppp$ of $\ppp$ and, respectively, $\tilde \rrr$ of $\rrr$ with respect to $y_0$.

\hspace{\itemindent} Moreover, the closed embedding $\Phi: \xg \into \pp^m(\kk)$ is induced by the surjective homomorphism $\Phi^*: \kk[y_0,\ldots, y_m] \to \ag$ which maps $y_0 \mapsto (1)_1$ and $y_i \mapsto (g_i)_1$ for $1 \leq i \leq m$. Therefore ${\overline V}^\scrG$ in $\xg$ is defined by the ideal $\Phi^*(\tilde \rrr)$. But the $d$-th graded component of $\Phi^*(\tilde \rrr)$ is 
\begin{align*}
\Phi^*((\tilde \rrr)_d) &:= \{\Phi^*(\tilde f(y_0, \ldots, y_m)): \tilde f \in \tilde \rrr,\ \tilde f\ \text{homogeneous,}\ \deg(\tilde f) =d \} \\
										&= \{\tilde f((1)_1, (g_1)_{1}, \ldots, (g_m)_{1}): \\
										&\phantom{= \{} \tilde f\ \text{homogeneous in}\ y_0, \ldots, y_m,\ \deg(\tilde f) =d,\ \tilde f(1,y_1, \ldots, y_m) \in \rrr\}\\
										&= \{(\tilde f(1, g_1, \ldots, g_m))_d: \\
										&\phantom{= \{} \tilde f\ \text{homogeneous in}\ y_0, \ldots, y_m,\ \deg(\tilde f) =d, \tilde f(1,g_1, \ldots, g_m) \in \qqq\}\\
										&= \{(f(g_1, \ldots, g_m))_d: \\
										&\phantom{= \{} f\ \text{polynomial in}\ y_1, \ldots, y_m,\ \deg(f) \leq d,\ f(g_1, \ldots, g_m) \in \qqq\}\\
										&= \{(g)_d: g \in \qqq \cap G_d\},
\end{align*}
where the last equality is a consequence of $\ag = \kk[(1)_1, (g_1)_1, \ldots, (g_m)_1]$. Then $\Phi^*(\tilde \rrr) =$ $\dsum_{d \geq 0} \Phi^*((\tilde \rrr)_d) = \dsum_{d \geq 0} \{(g)_d: g \in \qqq \cap G_d\} = \qqq^\scrG$. Now recall (example \ref{proj-example3}) that $\xf$ and $\xg$ are isomorphic and this isomorphism is induced by the inclusion $\ag \subseteq \af$. Since $\qqq^\scrF \cap \ag = \qqq^\scrG$, it follows that the Zariski closed subset of $\xg$ determined by $\qqq^\scrG$ is isomorphic to the Zariski closed subset of $\xf$ determined by $\qqq^\scrF$.

\item $\psi_\scrF$ preserves the intersection of $V_1, \ldots, V_m$ at $\infty$ iff $\overline V_1 \cap \cdots \cap \overline{V}_m \cap X_\infty =  \emptyset$. By part \ref{homogeneous-closure} $\overline V_j = V(\qqq^\scrF_j)$ for each $j$, and by theorem \ref{completethm} $X_\infty = V((1)_1)$. Therefore $\overline V_1 \cap \cdots \cap \overline{V}_m \cap X_\infty$ is determined by the ideal of $\af$ generated by $(1)_1, \qqq^\scrF_1, \ldots, \qqq^\scrF_m$, which is precisely the definition of $\scrI$. Then the projective version of Nullstellensatz (see section \oldref{projective-subsection}) implies $V(\scrI) = \emptyset$ iff $\af_+ \subseteq \sqrt \scrI$. \qedhere
\end{asparaenum}
\end{proof}

\begin{thm}[see {\cite[Theorem 1.2(1)]{announcement}} and {\cite[Theorem 1.3.1]{preprint}}] \label{filtrexistence-thm1}
Let $V_1, \ldots, V_m$ be Zariski closed subsets in an affine variety $X$ such that $\cap_{i=1}^m V_i$ is a finite set. Then there is a complete filtration $\scrF$ on $\kk[X]$ such that $\psif$ preserves the intersection of the $V_i$'s at $\infty$.
\end{thm}

\begin{proof}
Let $X \subseteq \ank$ and the ideals in $\kk[x_1, \ldots, x_n]$ defining $X, V_1, \ldots, V_m$ be respectively $\ppp, \qqq_1, \ldots, \qqq_m$ with $\qqq_j \supseteq \ppp$ for each $j$.

\begin{claim*}
For each $i = 1, \ldots, n$, there is an integer $d_i \geq 1$ such that
\begin{align}
x_i^{d_i} &= f_{i,1} + \ldots + f_{i,m} + g_i \label{filtrexistence1-eqn1}
\end{align}
for some $f_{i,j} \in \qqq_j$ and a polynomial $g_i \in \kk[x_i]$ of degree less that $d_i$.
\end{claim*}

\begin{proof}
If $V_1 \cap \ldots \cap V_m = \emptyset$, then by Nullstellensatz $\langle \qqq_1, \ldots, \qqq_m \rangle$ is the unit ideal in $\kk[x_1, \ldots, x_n]$, and the claim is trivially satisfied with $g_i := 0$ for each $i$. So assume $$V_1 \cap \ldots \cap V_m = \{P_1, \ldots, P_k\} \subseteq \ank,$$
for some $k \geq 1$. Let $P_i = (a_{i,1}, \ldots, a_{i,n}) \in \ank$. For each $i = 1, \ldots, n$, let
$$h_i := (x_i - a_{1,i})(x_i - a_{2,i}) \cdots (x_i - a_{k,i}).$$
By Nullstellensatz, for some $d'_i \geq 1$, $h_i^{d'_i} \in \langle \qqq_1, \ldots, \qqq_m \rangle$, i.e. $h_i^{d'_i} = f_{i,1} + \ldots + f_{i,m}$ for some $f_{i,j} \in \qqq_j$. Substituting $h_i = \prod_j(x_i - a_{j,i})$ in the preceding equation we see that the claim holds with $d_i := kd'_i$.
\end{proof}

Below for  $S \subseteq \kk[X]$ we denote by $\kk\langle S \rangle$ the $\kk$-linear span of $S$, and for an element $g \in \kk[x_1, \ldots, x_n]$, we denote by $\bar g$ the image of $g$ in $\kk[X] = \kk[x_1, \ldots, x_n]/\ppp$. Fix a set of $f_{i,j}$'s satisfying the conclusion of the previous claim. Then define a filtration $\scrF$ on $\kk[X]$ as follows: let
\begin{align*}
	F_0 &:= \kk, \\
	F_1 &:= \kk\langle 1, \bar x_1, \ldots, \bar x_n, \bar f_{1,1}, \ldots, \bar f_{n,m} \rangle, \\
	F_k &:= F_1^k\ \textrm{for}\ k>1,\\
	\scrF &:= \{F_i:\ i\geq 0\}.
\end{align*}

Clearly $\scrF$ is a complete filtration. We now show that this $\scrF$ satisfies the conclusion of the theorem. By lemma \ref{preservation-criterion} this is equivalent to showing that $\sqrt \scrI \supseteq \kxf_+$, where $\scrI$ is the ideal generated by ${\bar \qqq_1}^\scrF, \ldots, {\bar \qqq_m}^\scrF$ and $(1)_1$ in $\kxf$.

\sloppy

From the construction of $\scrF$ it follows that $\kxf_+$ is generated by the elements $(1)_1, (\bar x_1)_1, \ldots, (\bar x_n)_1, (\bar f_{1,1})_1, \ldots, (\bar f_{n,m})_1$. Note that $\bar f_{i,j} \in \bar \qqq_j$ for each $i,j$, so that $(\bar f_{i,j})_1 \in {\bar \qqq_j}^\scrF \subseteq \scrI$. Moreover, $(1)_1 \in \scrI$. So, all we really need to show is that $(\bar x_i)_1 \in \sqrt \scrI$ for all $i = 1, \ldots, n$.

\fussy

Reducing equation \eqref{filtrexistence1-eqn1} mod $\ppp$, we have $(\bar x_i)^{d_i} = \bar f_{i,1} + \ldots + \bar f_{i,m} + \bar g_i \in \kk[X]$ for all $i = 1, \ldots, n$. Let $g_i = \sum_{j=0}^{d_i-1}a_{i,j}x_i^j$. Then in $\kxf$,
$$((\bar x_i)_1)^{d_i} = ((1)_1)^{d_i-1}((\bar f_{i,1})_1 + \ldots + (\bar f_{i,m})_1) + \sum_{j=0}^{d_i-1}a_{i,j}((\bar x_i)_1)^j((1)_1)^{d_i-j}.$$
All of the summands in the right hand side lie inside $\scrI$, hence $((\bar x_i)_1)^{d_i} \in \scrI$ for all $i = 1, \ldots, n$, as required.
\end{proof}

Recall that given a polynomial map $f=(f_1, \ldots, f_q):X \to \aqk$, $a = (a_1, \ldots, a_q) \in \aqk$ and a completion $\psi$ of $X$, $\psi$ {\em preserves $\{f_1, \ldots, f_n\}$ at $\infty$ over $a$} if $\psi$ preserves the intersection of the hypersurfaces $H_i(a) := \{x \in X: f_i(x) = a_i\}$, $i = 1, \ldots, q$.

\begin{example}\label{filtrexistence-example1}
Consider map $f:\affine{2}{\kk} \to \affine{2}{\kk}$ given by $f(x,y) := (x, y + x^3)$. For $a := (a_1, a_2) \in \kk^2$, 
\begin{align*}
H_1(a) &= \{(a_1, y): y \in \kk\},\\
H_2(a) &= \{(x,a_2-x^3):x \in \kk\}.
\end{align*}
We claim that in the usual completion $\pp^{2}(\kk)$ of $\kk^2$, the closures of $H_1(a)$ and $H_2(a)$ intersect at a point $P$ at infinity for each $a \in \kk^2$, and hence $\pp^{2}(\kk)$, as the natural completion of $\affine{2}{\kk}$, does not preserve $\{f_1, \ldots, f_n\}$ at $\infty$ over any point of $\affine{2}{\kk}$. 

Indeed, write the homogeneous coordinates of $\pp^2(\kk)$ as $[z:x:y]$ and identify $\affine{2}{\kk}$ with $\pp^2(\kk)\setminus V(z)$. Let $a \in \affine{2}{\kk}$. When $\kk = \cc$, the `infinite' points in $\overline H_i(a)$ can be described as the limits of points in $H_i(a)$. Therefore, the points at infinity of $H_1(a)$ are $\lim_{|y| \to \infty} [1:a_1:y] = \lim_{|y| \to \infty} [1/y:a_1/y:1] = [0:0:1]$. Similarly, the infinite part of $H_2(a)$ is $\lim_{|x| \to \infty} [1:x:a_2 - x^3] = \lim_{|x| \to \infty} [1/(a_2 - x^3):x/(a_2 - x^3):1] = [0:0:1]$, and hence the claim is true with $P := [0:0:1]$. 

To verify the claim for an arbitrary $\kk$, one has to apply lemma \ref{preservation-criterion} with $X = \affine{2}{\kk}$ and calculate $\overline H_i(a) \cap X_\infty = V(\qqq_i^\scrF(a), (1)_1)$, where $\qqq_i(a)$ is the ideal of $H_i(a)$. A straightforward calculation shows: the graded ring $\kxf$ corresponding to the embedding $\affine{2}{\kk} \into \pp^2(\kk)$ is isomorphic to $\kk[x,y,z]$ where $z$ plays the role of $(1)_1$, and $\qqq_i^\scrF(a)$ is the homogenization of $\qqq_i(a)$ with respect to $z$. Then $\qqq_1(a) = \langle x - a_1 \rangle$ and $\qqq_2(a) = \langle y+ x^3 - a_2 \rangle$, so that $\qqq_1^\scrF(a) = \langle x - a_1z \rangle$ and $\qqq_2^\scrF(a) = \langle yz^2+ x^3 - a_2z^3 \rangle$. Therefore $\overline H_1(a) \cap X_\infty = V(x - a_1z, z) = V(x,z) = \{[0:0:1]\}$. Similarly $\overline H_2(a) \cap X_\infty = V(yz^2 +x^3 - a_2z, z) = V(x,z) = \{[0:0:1]\}$ and the claim is valid with the same $P$ as in $\kk= \cc$ case.

Because it is simpler to describe, we will from now on frequently use only the limit argument (valid only for $\kk = \cc$) in order to find the points at infinity of various subvarieties of a given $X$. In all these cases, the analogous results also follow over an arbitrary algebraically closed field $\kk$ by means of straightforward calculations (and if $\Char \kk = 0$, by Tarski-Lefschetz principle).

We now find, following the proof of Theorem \ref{filtrexistence-thm1}, a completion of $\affine{2}{\kk}$ which preserves $\{f_1, \ldots, f_n\}$ at $\infty$ over $0$. In the notation of theorem \ref{filtrexistence-thm1}, $\qqq_1	= \langle x \rangle$ and $\qqq_2	= \langle y+x^3 \rangle$. Observe that $x \in \qqq_1$, and $y$ satisfies
\begin{align*}
y =  - x^3 + (y+x^3),
\end{align*}
with $x^3 \in \qqq_1$ and $y+x^3 \in \qqq_2$. Let filtration $\scrF:= \{F_i: i\geq 0\}$ on $\kk[x,y]$ be defined as follows: $F_0 := \kk$, $F_1 := \kk\langle 1, x, y, x^3 \rangle$, and $F_k := (F_1)^k$ for $k>1$. Then as in the proof of theorem \ref{filtrexistence-thm1}, completion $\psi_\scrF$ preserves $\{f_1, \ldots, f_n\}$ at $\infty$ over $0$. Let us now show this directly. By corollary \ref{completecor}, the corresponding completion $\xf$ is isomorphic to the closure in $\pp^3(\kk)$ of the image of $\phi: \affine{2}{\kk} \to \pp^3(\kk)$, where $\phi(x,y) := [1:x:y:x^3]$. Then $\phi(H_1(a)) = \{[1:a_1:y:a_1^3]: y \in \kk\}$ and limit $\lim_{y \to \infty}[1:a_1:y:a_1^3] =$ $\lim_{y \to \infty}[1/y:a_1/y:1:a_1^3/y] = [0:0:1:0]$, for $a \in \affine{2}{\kk}$, so that the only point at infinity of $\overline H_1(a)$ is $[0:0:1:0]$. Similarly, $\phi(H_1(a)) = \{[1:x:a_2 - x^3:x^3]: x \in \kk\}$ and $\lim_{x \to \infty}[1:x:a_2 - x^3:x^3] = \lim_{x \to \infty}[1/x^3:1/x^2:(a_2 - x^3)/x^3:1] = [0:0:-1:1]$. Therefore $\overline H_2(a)$ also has only one point at infinity and it is $[0:0:-1:1]$. It follows that $\overline H_1(a) \cap \overline H_2(a) \cap X_\infty = \emptyset$ for all $a$, i.e. $\xf$ preserves $\{f_1, \ldots, f_n\}$ at $\infty$ over {\em every} point of $\affine{2}{\kk}$.
\end{example}

\begin{example}\label{filtrexistence-example2}
Let $f(x,y) := (x,y)$ on $\kk^2$. Then for each $a = (a_1, a_2) \in \kk^2$, $H_1(a) = \{(a_1,y): y \in \kk\}$ and $H_2(a) = \{(x,a_2): x \in \kk\}$. Consider filtration $\scrF$ on $\kk[x,y]$ defined by: $F_0 := \kk$, $F_1 := \kk\langle 1, x, y, xy, x^2y^2 \rangle$, and $F_k := (F_1)^k$ for $k \geq 2$. By corollary \ref{completecor}, $\xf$ is the closure of the image of $\affine{2}{\kk}$ under the map $\phi:\affine{2}{\kk} \into \pp^4(\kk)$ defined by: $\phi(x,y) = [1:x:y:xy:x^2y^2]$. Then $\phi(H_1(a)) = \{[1:a_1:y:a_1 y:a_1^2y^2]: y \in \kk\}$. If $a_1=0$, then $\phi(H_1(a)) = \{[1:0:y:0:0]: y \in \kk\}$, and hence the only point at infinity in $\overline{\phi(H_1(a))}$ is $[0:0:1:0:0]$. But if $a_1 \neq 0$, then dividing all coordinates by $a_1^2y^2$, we see that the point at infinity in $\overline{\phi(H_1(a))}$ is $[0:0:0:0:1]$. Similarly, $\phi(H_2(a)) = \{[1:x:a_2:y:a_2 x:a_2^2x^2]: x \in \kk\}$ and the only point at infinity in $\overline{\phi(H_2(a))}$ is $[0:1:0:0:0]$ if $a_2=0$, and $[0:0:0:0:1]$ if $a_2 \neq 0$. Therefore $\xf$ preserves $\{f_1, \ldots, f_n\}$ at $\infty$ over $a$ iff $a_1 = 0$ or $a_2 = 0$, i.e. iff $a$ belongs to the union of the coordinate axes.
\end{example}

Let $f: X \to Y$ be a generically finite map of affine varieties of the same dimension. Given any $y \in Y$ such that $\finv(y)$ is finite, theorem \ref{filtrexistence-thm1} guarantees the existence of a projective completion of $X$ that preserves $\{f_1, \ldots, f_n\}$ at $\infty$ over $y$. But as the preceding example shows, the completion might fail to preserve  $\{f_1, \ldots, f_n\}$ at $\infty$ over `most of the' points in the image of $f$. This suggests that we should look for a completion $\psi$ which preserves $\{f_1, \ldots, f_n\}$ at $\infty$ over $y$ for {\em generic} $y \in Y$, i.e. $\psi$ {\em preserves $\{f_1, \ldots, f_n\}$ at $\infty$}. We will demonstrate two ways to accomplish this goal - we start with a simpler-to-prove theorem \ref{filtrexistence-thm2} and will present a stronger version in theorem \ref{filtrexistence-thm3} following (cf. \cite[Theorem 1.2]{announcement} and \cite[Theorem 1.3.4]{preprint}).

\begin{thm}\footnote{The idea of looking at the construction of theorem \ref{filtrexistence-thm2} is due to Professor A. Khovanskii.} \label{filtrexistence-thm2}
Let $f: X \to Y \subseteq \aqk$ be a generically finite map of affine varieties of same dimension. Include $\aqk$ into $(\pp^1(\kk))^q$ via the componentwise inclusion $(a_1, \ldots, a_q) \mapsto ([1:a_1], \ldots, [1: a_q])$. Let $\phi: X \into Z$ be any completion of $X$. Define $\bar X$ to be the closure of the graph of $f$ in $Z \times (\pp^1(\kk))^q$. Then $\bar X$ preserves $\{f_1, \ldots, f_n\}$ at $\infty$. If $\phi$ comes from some filtration on $\kk[X]$, then there is a filtration $\scrF$ on $\kk[X]$ and a commutating diagram as follows:
$$\xymatrix@R-1pc@C-1pc{
						& X \ar[1,-1] \ar[1,1] 	&\\
\bar X \ar[0,2]^\cong 	& 						& \xf}$$
\end{thm}

\begin{proof}
Let $\pi:= (\pi_1, \ldots, \pi_q): Z \times (\pp^1(\kk))^q \to (\pp^1(\kk))^q$ be the natural projection. Then $\pi$ maps $\bar X$ onto the closure $\bar Y$ of $Y$ in $(\pp^1(\kk))^q$. Let $V := Z \setminus X$. Then $\tilde V := (V \times (\pp^1(\kk))^q) \cap \bar X$ is a proper Zariski closed subset of $\bar X$. Since $Z$ is complete, it follows that $\pi(\tilde V)$ is a proper Zariski closed subset of $\bar Y$. We now show that for all $y \in Y \setminus \pi(\tilde V)$, $\bar X$ preserves $\{f_1, \ldots, f_n\}$ at $\infty$ over $y$.

Pick an arbitrary $y:=(y_1, \ldots, y_q) \in Y$ such that $\bar X$ does not preserve $\{f_1, \ldots, f_n\}$ at $\infty$ over $y$. It suffices to show that $y \in \pi(\tilde V)$. As usual, let $H_i(y) := \{x \in X: f_i(x) = y_i\}$. By assumption there is $\tilde x \in \bar H_1(y) \cap \cdots \cap \bar H_q(y) \cap (\bar X \setminus X)$, where for each $k$, $\bar H_k(y)$ is the closure of $H_k(y)$ in $\bar X$. Fix a $k$, $1 \leq k \leq q$. Note that $\pi_k(H_k(y)) = \{y_k\}$. By continuity of $\pi_k$ it follows that $\pi_k(\bar H_k(y)) = \{y_k\}$. But then $\pi_k(\tilde x) = y_k$. It follows that $\pi(\tilde x) = y$ and hence $\tilde x$ is of the form $(z,y)$ for some $z \in Z$. We claim that $z$ does not lie in $X$. Indeed, if $z \in X$, it would imply $\tilde x \in (X \times Y) \cap (\bar X \setminus \psi(X))$, where $\psi: X \into \bar X$ is the inclusion. Consider the chain of inclusions: $\psi(X) \subseteq X \times Y \subseteq Z \times (\pp^1(\kk))^q$. Note: 

\begin{asparaenum}
\item $\psi(X)$ is the graph of $f$ in $X \times Y$, and hence is a Zariski closed subvariety of $X \times Y$.
\item $X \times Y$ is Zariski {\em open} in $Z \times (\pp^1(\kk))^q$.
\item If $T \subseteq U \subseteq W$ are topological spaces such that $T$ is closed in $U$ and $U$ is open in $W$, then $\bar T \cap U = T$, where $\bar T$ is the closure of $T$ in $W$. 
\end{asparaenum}

Since $\bar X$ is by definition the closure of $\psi(X)$ in $Z \times (\pp^1(\kk))^q$, it follows via the above observations that $\bar X \cap (X \times Y) = \psi(X)$, so that $(X \times Y) \cap (\bar X \setminus \psi(X)) = \emptyset$. This contradiction proves the claim. It follows that $z \in Z\setminus X = V$. Then $\tilde x \in \tilde V$. Therefore $y = \pi(\tilde x) \in \pi(\tilde V)$ and the first claim of the theorem is proved. 

As for the last claim, note that if $\phi$ comes from a filtration, then by corollary \ref{completecor} we may assume that $Z \subseteq \pp^p(\kk)$ for some $p$ and $\phi(x) = [1:g_1(x): \cdots : g_p(x)]$ for some $g_1, \ldots, g_p \in \kk[X]$. Hence the inclusion $\psi:X \into \bar X$ is of the form: $\psi(x) = ([1:g_1(x): \cdots : g_p(x)],[1:f_1(x)], \ldots, [1:f_q(x)])$. Let $l := (p+1)2^q - 1$ and let us embed $\pp^p(\kk) \times (\pp^1(\kk))^q \into \pp^l(\kk)$ via the {\em Segre} embedding $s$ which maps $w := ([w_0: \cdots : w_p], [w_{1,0}: w_{1,1}], \ldots, [w_{q,0}: w_{q,1}])$ to the point $s(w)$ whose homogeneous coordinates are monomials of degree $q+1$ in $w$ of the form $w_iw_{1,j_1}w_{2,j_2}\cdots w_{q,j_q}$ where $0 \leq i \leq p$ and $0 \leq j_k \leq 1$ for each $k$. The component of $s\circ \psi$ corresponding to $i = j_1 = \cdots = j_q = 0$ is $1$ and hence $s \circ \psi$ maps $x \in X$ to a point with homogeneous coordinates $[1: h_1(x): \cdots : h_l(x)]$ for some $h_1, \ldots, h_l \in \kk[X]$. Then corollary \ref{completecor} implies that there is a filtration $\scrF$ on $\kk[X]$ such that the closure $\bar X$ of $s \circ \psi(X)$ in $\pp^l(\kk)$ is isomorphic to $\xf$ via an isomorphism which is identity on $X$. Morphism $s$ being an isomorphism completes the proof.
\end{proof}

\begin{rem} \label{preserved-set}
Let $f: X \to Y$ be a map of $n$-dimensional affine varieties with generically finite fibers and $\psi$ be a completion of $X$. Define $S_\psi := \{a \in f(X): \psi$ preserves $\{f_1, \ldots, f_n\}$ at $\infty$ over $y\}$. It will be interesting to know if $S_\psi$ has any intrinsic structure. In example \ref{filtrexistence-example2} $S_\psi$ was the union of two coordinate axes in $\affine{2}{\kk}$, and hence a proper closed subset of $f(X)$. On the other hand, theorem \ref{filtrexistence-thm2} shows that there are completions $\psi$ of $X$ such that $S_\psi$ contains a dense open subset of $f(X)$. We now give an example where $S_\psi$ is indeed a proper dense open subset of $f(X)$, namely:

Let $X = Y = \affine{2}{\cc}$ and $f:X \to Y$ be the map defined by $f_1 := x_1^3 + x_1^2x_2 + x_1x_2^2 - x_2$ and $f_2 := x_1^3 + 2x_1^2x_2 + x_1x_2^2 - x_2$. It is easy to see that $f$ is quasifinite. Let $\phi:X \into \pp^2(\cc)$ be the usual completion, and let $\psi: X \into \bar X$ be as in theorem \ref{filtrexistence-example2}. Let the coordinates of $\pp^2(\cc)$ be $[Z:X_1:X_2]$. Identify $X$ with $\pp^2(\cc) \setminus V(Z)$, so that $x_i = X_i/Z$ for $i = 1,2$. Then $X \ni (x_1,x_2) \overset{\psi}{\mapsto} ([1:x_1:x_2], [1:f_1(x)], [1:f_2(x)])$. We claim that $f(X)\setminus S_\psi$ is the line $L :=\{(c,c): c \in \cc\}$.

Indeed, let $a := (a_1, a_2) \in Y$. Define, as usual, $H_i(a) := \{x \in X: f_i(x) = a_i\}$ for $i = 1,2$. Let $C_i(a)$ be the closure in $\pp^2(\cc)$ of $H_i(a)$ for each $i$. It is easy to see that $P := [0:0:1] \in C_1(a) \cap C_2(a)$. Choose local coordinates $\xi_1 := X_1/X_2$ and $\xi_2 := Z/X_2$ of $\pp^2(\cc)$ near $P := [0:0:1]$. Equations of $f_{i,a}$ in $(\xi_1,\xi_2)$ coordinates are:
\begin{align}
\begin{split} \label{in-xi-coords}
f_{1,a} &= \xi_1^3 + \xi_1^2 + \xi_1 - \xi_2^2 - a_1\xi_2^3 \\
f_{2,a} &= \xi_1^3 + 2\xi_1^2 + \xi_1 - \xi_2^2 - a_2\xi_2^3
\end{split}
\end{align}
It follows that for each $i$, $C_i(a)$ is smooth at $P$ (in particular, each has only one branch at $P$) and both admit parametrizations at $P$ of the form
\begin{align} \label{parametrizations} 
\gamma_{i,a}(t) := [t:t^2 + \littleoh(t^3):1],
\end{align}
where $\littleoh(t^3)$ means terms of order $t^3$ and higher. In $(x_1, x_2)$ coordinates the parametrizations are of the form: $(x_1(t), x_2(t)) = (t +\littleoh(t^2), 1/t)$ for $t \neq 0$. Since $f_2(x) = f_1(x) + x_1^2x_2$, it follows that for $t \neq 0$, 
\begin{align*}
\psi(\gamma_{1,a}(t)) &= (\gamma_{1,a}(t), [1:a_1], [1: a_1 + \frac{(t +\littleoh(t^2))^2}{t}]) \\
											&= (\gamma_{1,a}(t), [1:a_1], [1: a_1 + t +\littleoh(t^2)])   
\end{align*}
Therefore $\lim_{t \to 0} \psi(\gamma_{1,a}(t)) = (P, [1:a_1], [1:a_1])$. Since $f_1(x) = f_2(x) - x_1^2x_2$, the same argument also gives $\lim_{t \to 0} \psi(\gamma_{2,a}(t)) = (P, [1:a_2], [1:a_2])$. 

To summarize, we proved that if $a \in L$, then $(P,[1:a_1], [1:a_1]) \in \bar H_1(a) \cap \bar H_2(a) \cap X_\infty$, where as usual $\bar H_i(a)$ is the closure of $H_i(a)$ in $\bar X$ and $X_\infty := \bar X \setminus X$. It follows that $L \subseteq f(X)\setminus S_\psi$. 

To prove the other inclusion, assume $a \in f(X)\setminus S_\psi$. Pick $z \in \bar H_1(a) \cap \bar H_2(a)\cap X_\infty$. Then $z = (Q, [1:a_1], [1:a_2])$ for a point $Q \in \pp^2(\cc)$. Therefore it follows that $Q \in (C_1(a) \setminus H_1(a)) \cap (C_2(a) \setminus H_2(a))$, where curves $C_i(a)$ are the closures in $\pp^2(\cc)$ of $H_i(a)$, $i=1,2$. But the only possible choice for such point $Q$ is point $P$. Therefore limit $\lim_{t \to 0} \psi(\gamma_{1,a}(t)) = z = \lim_{t \to 0} \psi(\gamma_{2,a}(t))$, which implies that $(P, [1:a_1], [1:a_1]) =$ $(P, [1:a_2], [1:a_2])$. Therefore $a_1 = a_2$ and $a \in L$, as claimed.
\end{rem}

Let $f:X \to Y \subseteq \aqk$ be as in theorem \ref{filtrexistence-thm2}. In the theorem following we find completions with even stronger preservation property at $\infty$, namely completions that {\em preserve map $f$ at $\infty$} (remark-definition \ref{preserve-map-defn}).

\begin{thm}[cf. {\cite[Theorem 1.2(2)]{announcement}} and {\cite[Theorem 1.3.4]{preprint}}]\label{filtrexistence-thm3}
Let $f: X \to Y \subseteq \aqk$ be a generically finite map of affine varieties of the same dimension. Then there is a complete filtration $\scrF$ on the coordinate ring of $X$ such that $\psi_\scrF$ preserves map $f$ at $\infty$.
\end{thm}

\begin{proof}
Choose a set of coordinates $x_1, \ldots, x_p$ (resp. $y_1, \ldots, y_q$) of $X$ (resp. $Y$). Since $f$ is generically finite, it follows that the coordinate ring of $X$ is algebraic over the pullback of the coordinate ring of $Y$. In particular, each $x_i$ satisfies a polynomial of the form: 
\begin{align}
\sum_{j=0}^{k_i}g_{i,j}(y)(x_i)^j &= 0 \label{app1-eqn2}
\end{align}
for some $k_i \geq 1$ and regular functions $g_{i,j}(y)$ on $Y$ such that $g_{i,k_i} \neq 0 \in \kk[Y]$. In abuse of notation, but for the sake of convenience, we implicitly identified in \eqref{app1-eqn2} variables $y_k$ with polynomials $f_k$ for each $k$. We continue to do so throughout this proof. Let $g_{i,j}(y) = \sum_{\alpha}c_{i,j,\alpha}y^{\alpha}$ be an arbitrary representation of $g_{i,j}$ in  $\kk[Y]$. For each $i,j$ with $1 \leq i \leq p$ and $0 \leq j \leq k_i$, let $d_{i,j} := \deg_y(\sum_{\alpha}c_{i,j,\alpha}y^{\alpha}) := \max\{|\alpha|: c_{i,j,\alpha} \neq 0\}$, where $\alpha = (\alpha_1, \ldots, \alpha_q) \in (\zz_+)^q$ and $|\alpha| := \alpha_1 + \cdots + \alpha_q$. Let $d_0 := \max \{d_{i,0} : 1 \leq i \leq p\}$ and $k_0 := \max \{k_i : 1 \leq i \leq p\}$. Define a filtration $\scrF := \{F_i: i\geq 0\}$ on $\kk[X]$ as follows:
\begin{eqnarray*}
	F_0 &:=& \kk, \\
	F_1 &:=& \kk\langle 1, x_1, \ldots, x_p, y_1, \ldots, y_q \rangle\ +\ \kk\langle y^{\beta}: |\beta| \leq d_0 \rangle\ + \\
		&  &\ +\ \kk\langle x_iy^{\beta}: |\beta| \leq d_{i,1}, 1 \leq i \leq p \rangle, \\
	F_k &:=& \left\{
		\begin{array}{l}
		\sum_{j=1}^{k-1} F_jF_{k-j} + \kk\langle (x_i)^ky^{\beta}: |\beta| \leq d_{i,k}, 1 \leq i \leq p \rangle\quad \textrm{if}\ 1<k \leq k_0,\\
		\sum_{j=1}^{k-1} F_jF_{k-j} \quad \textrm{if}\ k>k_i\ \forall i.
		\end{array}
		\right.
\end{eqnarray*}

Let $g := \prod_{i=1}^m g_{i,k_i}$ and $U :=  \{a \in Y: g(a) \neq 0\}$. Then $U$ is a non-empty Zariski open subset of $Y$. Let $\xi := (\xi_1, \ldots, \xi_q): \aqk \to \aqk$ be an arbitrary linear change of coordinates of $\aqk$. It suffices to show that $\psi_\scrF$ preserves the components of $\xi \circ f$ at $\infty$ over $\xi(a)$ for $a:=(a_1, \ldots, a_q) \in U$. For each $a \in Y$, let $H_j(a) := \{x \in X: (\xi_j \circ f)(x) = \xi_j(a)\}$ and let $\qqq_j(a)$ be the ideal of $H_j(a)$, i.e. the ideal of $\kk[X]$ generated by $\xi_j(y) - \xi_j(a)$. By lemma \ref{preservation-criterion}, $\psi_\scrF$ preserves the components of $\xi \circ f$ at infinity over $\xi(a)$ iff $\sqrt{\scrI(a)} = \profing_+$, where $\scrI(a)$ is the ideal of $\kxf$ generated by $\qqq_1^\scrF(a), \ldots, \qqq_n^\scrF(a)$ and $(1)_1$. Note the following:

\begin{compactenum}[(a)]
\labelformat{enumi}{(\theenumi)}
\item \label{linearly-complete-1} Since $\xi$ is a linear change of coordinate, so is $\xi^{-1}$. Therefore, for all $d \geq 0$, the $\kk$-span of $\{y^\beta: |\beta| \leq d\}$ in $\kk[Y]$ is equal to the $\kk$-span of $\{(\xi^{-1}(y))^\beta: \deg_y((\xi^{-1}(y))^\beta) \leq d \}$.
\item \label{linearly-complete-2} If we replace $f$ by $\xi \circ f$, and hence $y$ by $\xi(y)$, then $g_{i,j}(y)$ changes to $g^\xi_{i,j}(y) := g_{i,j}(\xi^{-1}(y)) = \sum_{\alpha} c_{i,j,\alpha} (\xi^{-1}(y))^{\alpha}$. But replacing $\sum_\alpha c_{i,j,\alpha} y^{\alpha}$ by $\sum_{\alpha} c_{i,j,\alpha} (\xi^{-1}(y))^{\alpha}$ does {\em not} change its degree $d_{i,j}$ in $y$.
\item \label{linearly-complete-3} Let $g^\xi := \prod_{i=1}^m g^\xi_{i,k_i}$. Then $g^\xi(\xi(a)) \neq 0$ iff $g(a) \neq 0$.
\end{compactenum}

In view of the latter observations and the construction of $\scrF$ it follows that $\scrF$ does not change if we replace $f$ by $\xi \circ f$. Moreover, the following two claims are equivalent due to properties \ref{linearly-complete-1}, \ref{linearly-complete-2} and \ref{linearly-complete-3} of the preceding paragraph.

\begin{compactenum}[(1)]
\labelformat{enumi}{(\theenumi)}
\item \label{cond-for-id} $\psi_\scrF$ preserves the components of $f$ at $\infty$ over $a \in U$, and
\item $\psi_\scrF$ preserves the components of $\xi \circ f$ at $\infty$ over $a \in \xi^{-1}(U)$.
\end{compactenum}

Therefore it suffices to prove \ref{cond-for-id} and we may without loss of generality assume $\xi$ to be the identity. Note that $\profing_+$ is generated as a $\kk$-algebra by elements $(1)_1$, $(x_1)_1, \ldots, (x_p)_1$, $(y_1)_1, \ldots, (y_q)_1$, the $(y^{\beta})_1$'s that appear in the definition of $F_1$, and all those $((x_i)^ky^\beta)_k$ that we inserted in the definition of all $F_k$'s. Therefore $\sqrt{\scrI(a)} = \profing_+$ iff some power of each of these generators lies in $\scrI(a)$.

\begin{prolemma} \label{powers-in-y-a}
Let $a$ be an arbitrary point in $Y$.
\begin{compactenum}
\item \label{powers-of-y-in-F1} Let $\beta \in (\zz_+)^q$ be such that $y^\beta \in F_1$. then
	\begin{compactenum}
	\item \label{powers-of-y-in-F1-p1} $(y-a)^\beta$ also lies in $F_1$, and
	\item \label{powers-of-y-in-F1-p2} $((y-a)^\beta)_1 \in \scrI(a)$.
	\end{compactenum}
\item \label{powers-of-x-in-Fk} Let $1 \leq i \leq p$. Pick $k$ with $1 \leq k \leq k_i$ and $\beta \in (\zz_+)^q$ such that $x_i^ky^\beta \in F_k$. Then 
	\begin{compactenum}
	\item \label{powers-of-x-in-Fk-p1} $x_i^k(y-a)^\beta$ lies in $F_k$, and
	\item \label{powers-of-x-in-Fk-p2} if in addition $\beta \neq 0$, then $(x_i^k(y-a)^\beta)_k \in \scrI(a)$.
	\end{compactenum}	 
\end{compactenum}
\end{prolemma}

\begin{proof}
\begin{asparaenum} 
\item Pick $\beta \in (\zz_+)^q$ such that $y^\beta \in F_1$. Expanding $(y-a)^\beta$ in powers of $y_1, \ldots, y_q$, we see that $(y-a)^\beta = \sum_{|\gamma| \leq |\beta|}c_\gamma y^\gamma$ for some $c_\gamma \in \kk$ and $\gamma \in (\zz_+)^q$. By construction, $F_1$ contains each of the $y^\gamma$ appearing in the preceding expression. It follows that $F_1$ also contains $(y-a)^\beta$, which proves assertion \ref{powers-of-y-in-F1-p1}. As for \ref{powers-of-y-in-F1-p2}, note that if $\beta = 0$, then $((y-a)^\beta)_1 = (1)_1 \in \scrI(a)$. Otherwise, there exists $j$, $1 \leq j \leq q$, such that the $j$-th coordinate of $\beta$ is positive. Then $(y-a)^\beta \in \qqq_j(a)$, and hence $((y-a)^\beta)_1 \in \qqq_j^\scrF(a) \subseteq \scrI(a)$, which completes the proof of assertion \ref{powers-of-y-in-F1}.

\item Let $1 \leq i \leq p$. Pick $k, \beta$ such that $1 \leq k \leq k_i$ and $x_i^ky^\beta \in F_k$. As in the proof of assertion \ref{powers-of-y-in-F1}, expanding $(y-a)^\beta$ in powers of $y_1, \ldots, y_q$, we see that $x_i^k(y-a)^\beta = \sum_{|\gamma| \leq |\beta|}c_\gamma x_i^k y^\gamma$ for some $c_\gamma \in \kk$. By construction, $F_k$ contains $x_i^ky^\gamma$ for each $|\gamma| \leq |\beta|$. It follows that $F_k$ also contains $x_i^k(y-a)^\beta$, and this proves assertion \ref{powers-of-x-in-Fk-p1}. For \ref{powers-of-x-in-Fk-p2}, note that if $\beta \neq 0$, then there exists $j$, $1 \leq j \leq q$, such that the $j$-th coordinate of $\beta$ is positive. Then $x_i^k(y-a)^\beta \in \qqq_j(a)$, and hence $(x_i^k(y-a)^\beta)_k \in \qqq_j^\scrF(a) \subseteq \scrI(a)$, which completes the proof of the lemma. \qedhere
\end{asparaenum}
\end{proof}

We now return to the proof of theorem \ref{filtrexistence-thm3}. Let $a$ be any point in $Y$ and $\beta$ be such that $y^\beta \in F_1$. Expanding $y^\beta$ in powers of $y_1 - a_1, \ldots, y_q - a_q$, we see that $y^\beta = \sum_{|\gamma| \leq |\beta|}c'_\gamma (y-a)^\gamma$ for some $c'_\gamma \in \kk$. By assertion \ref{powers-of-y-in-F1-p1} of lemma \ref{powers-in-y-a}, each $(y-a)^\gamma$ in the preceding expression lies in $F_1$. Therefore, in $\kxf$ element $(y^\beta)_1 = \sum_{|\gamma| \leq |\beta|}c'_\gamma ((y-a)^\gamma)_1$. But assertion \ref{powers-of-y-in-F1-p2} of lemma \ref{powers-in-y-a} implies that $((y-a)^\gamma)_1 \in \scrI(a)$ for all $|\gamma| \leq |\beta|$. Therefore $(y^\beta)_1 \in \scrI(a)$.

Now expand the polynomial of the left hand side of equation \eqref{app1-eqn2} (as a polynomial in $y := (y_1, \ldots, y_q)$) in powers of $y_1 - a_1, \ldots, y_q - a_q$. This leads to an equation of the form
\begin{align}\label{app1-eqn2'}
\sum_{j=0}^{k_i}g_{i,j}(a)x_i^j + \sum_{\underset{\beta \neq 0}{j \leq k_i}} h_{i,j,\beta}(a)(y-a)^\beta x_i^j &= 0\ , \tag{$\oldref{app1-eqn2}'$} 
\end{align}
where terms $(y-a)^\beta x_i^j$, for $\beta \neq 0$ and $j \leq k_i$, that appear in \eqref{app1-eqn2'} are such that $y^\beta x_i^j \in F_j$, and due to assertion \ref{powers-of-x-in-Fk-p1} of lemma \ref{powers-in-y-a}, $(y-a)^\beta x_i^j$ are in $F_j \subseteq F_{k_i}$. Therefore equation \eqref{app1-eqn2'} implies the following equality in $\kxf$:
\begin{align}\label{app1-eqn2''}
g_{i,k_i}(a)((x_i)_1)^{k_i} = - \sum_{j=0}^{k_i-1}g_{i,j}(a)((x_i)_1)^j((1)_1)^{k_i-j} \ - 
        	 \sum_{\underset{\beta \neq 0}{ j \leq k_i}} h_{i,j,\beta}(a) ((y-a)^\beta x_i^j)_{k_i}. \tag{$\oldref{app1-eqn2}''$}
\end{align}
Since $(1)_1 \in \scrI(a)$, each of the terms under the first summation of the right hand side of \eqref{app1-eqn2''} lies in $\scrI(a)$. Moreover, by assertion \ref{powers-of-x-in-Fk-p2} of lemma \ref{powers-in-y-a}, every $((y-a)^\beta x_i^j)_{k_i}$ under the second summation of the right hand side of \eqref{app1-eqn2''} also belongs to $\scrI(a)$. It follows that $g_{i,k_i}(a)((x_i)_1)^{k_i} \in \scrI(a)$, and therefore $(x_i)_1 \in \sqrt{\scrI(a)}$ if $g_{i,k_i}(a) \neq 0$. Hence for all $a \in U$ and $i$, $1 \leq i \leq p$, elements $(x_i)_1 \in \sqrt{\scrI(a)}$. 

Let $a \in U$ and $1 \leq i \leq p$. Pick $(x_i)^ky^\beta \in F_k$ with $1 \leq k \leq k_i$. Expanding $y^\beta$ in powers of $y_1 - a_1, \ldots, y_q - a_q$, we see that $(x_i)^ky^\beta = \sum_{|\gamma| \leq |\beta|}c'_\gamma (x_i)^k(y-a)^\gamma$ for some $c'_\gamma \in \kk$. By assertion \ref{powers-of-x-in-Fk-p1} of lemma \ref{powers-in-y-a}, each $(x_i)^k(y-a)^\gamma$ in the preceding expression lies in $F_k$. Therefore in $\kxf$ element $((x_i)^ky^\beta)_k = \sum_{|\gamma| \leq |\beta|}c'_\gamma ((x_i)^k(y-a)^\gamma)_k$. If $|\gamma| \leq |\beta|$ and $\gamma \neq 0$, assertion \ref{powers-of-x-in-Fk-p2} of lemma \ref{powers-in-y-a} implies that $((x_i)^k(y-a)^\gamma)_k \in \scrI(a)$ and if $\gamma = 0$, then $((x_i)^k(y-a)^\gamma)_k = ((x_i)^k)_k = ((x_i)_1)^k \in \sqrt{\scrI(a)}$ due to the conclusion of the preceding paragraph. Therefore for $x_i^ky^\beta \in F_k$, $1 \leq k \leq k_i$, it follows that $((x_i)^ky^\beta)_k \in \sqrt{\scrI(a)}$. 

Consequently $\sqrt{\scrI(a)} = \kxf_+$ for all $a \in U$, which completes the proof.
\end{proof}

\begin{example} \label{strong-preservation}
Let $X = Y = \affine{2}{\cc}$ and $f:=(f_1, f_2):X \to Y$ be the map defined by $f_1 := x_1^3 + x_1^2x_2 + x_1x_2^2 - cx_2$ and $f_2 := x_1^3 + x_1^2x_2 + x_1x_2^2 - x_2$ for any complex number $c \neq 0$ or $1$. Note that this map is a minor variation of the map considered in remark \ref{preserved-set}. Let $\psi: X \into \bar X \subseteq \pp^2(\cc) \times \pp^1(\cc) \times \pp^1(\cc)$ be the completion considered in theorem \ref{filtrexistence-thm2} and remark \ref{preserved-set}, i.e. $\psi(x_1, x_2) := ([1:x_1:x_2],[1:f_1(x_1, x_2)], [1: f_2(x_1, x_2)])$ for all $(x_1, x_2) \in X$. Below we show that $\psi$ does {\em not} satisfy the preservation property of theorem \ref{filtrexistence-thm3}. 

For each $\lambda := (\lambda_1, \lambda_2) \in \cc^2$ let 
$$f_{\lambda} := \lambda_1 f_1 + \lambda_2 f_2 =  (\lambda_1 + \lambda_2)x_1^3 + (\lambda_1 + \lambda_2)x_1^2x_2 + (\lambda_1+\lambda_2)x_1x_2^2 - (\lambda_1 c+\lambda_2)x_2\: .$$
Fix $\lambda^1, \lambda^2 \in \cc^2\setminus (L_1 \cup L_2 \cup L_3)$ where $L_1$ is the $x_1$-axis, $L_2$ is the $x_2$-axis and $L_3$ is the line $\{(x_1, x_2) \in \affine{2}{\cc}: x_1 + x_2 = 0\}$. Fix an $a := (a_1, a_2) \in Y$. Let $b_i := \lambda^i_1a_1 + \lambda^i_2a_2$, $i=1,2$. Every $H_i(a) = \{x \in X: f_{\lambda^i}(x) = b_i\}$ is as in theorem \ref{filtrexistence-thm3}, $i = 1,2$. Let $C_i(a)$ be the closure in $\pp^2(\cc)$ of $H_i(a)$ for $i = 1,2$. Then $P := [0:0:1] \in C_1(a) \cap C_2(a)$, where as in remark \ref{preserved-set}, we choose coordinates $[Z:X_1:X_2]$ on $\pp^2(\cc)$ and identify $X$ with $\pp^2(\cc) \setminus V(Z)$. Choose local coordinates $\xi_1 := X_1/X_2$ and $\xi_2 := Z/X_2$ on $\pp^2(\cc)$ near $P := [0:0:1]$. Equations of $f_{\lambda^i}(x_1,x_2) - b_i$ in $(\xi_1,\xi_2)$ coordinates are:
\begin{align*}
f_{\lambda^i}(x_1,x_2) - b_i &= (\lambda^i_1 + \lambda^i_2)\xi_1^3 + (\lambda^i_1 + \lambda^i_2)\xi_1^2 + (\lambda^i_1+\lambda^i_2)\xi_1 - (\lambda^i_1c + \lambda^i_2)\xi_2^2 - a_i\xi_2^3\ .
\end{align*}
Since $\lambda^i_1 + \lambda^i_2 \neq 0$, it follows (similarly to the implication \eqref{in-xi-coords} $\im$ \eqref{parametrizations} of remark \ref{preserved-set}) that for each $i$, $C_i(a)$ is smooth at $P$ and has a parametrization at $P$ of the form
$$\gamma_{i,a}(t) := [t:\frac{\lambda^i_1c + \lambda^i_2}{\lambda^i_1+\lambda^i_2} t^2 + \littleoh(t^3):1]\ .$$
In $(x_1, x_2)$ coordinates the parametrizations are: $(x_1(t), x_2(t)) = (\frac{\lambda^i_1c + \lambda^i_2}{\lambda^i_1+\lambda^i_2} t +\littleoh(t^2), 1/t)$ for $t \in \cc^*$. A straightforward calculation shows that for $t \in \cc^*$
\begin{align}
\begin{split} \label{f-lambda-i}
f_1(\gamma_{i,a}(t)) &= \frac{\lambda^i_2(1-c)}{\lambda^i_1+\lambda^i_2}\frac{1}{t} + \littleoh(t^2)\quad \text{and}\\
f_2(\gamma_{i,a}(t)) &= \frac{\lambda^i_1(c-1)}{\lambda^i_1+\lambda^i_2}\frac{1}{t} + \littleoh(t^2)\: ,
\end{split} 
\end{align}
$i = 1, 2$. By our choice of $c \neq 0$, $1$ and $\lambda^i$'s off $L_1 \cup L_2 \cup L_3$ it follows that the coefficients at $\frac{1}{t}$ in the expressions in \eqref{f-lambda-i} are non-zero. It follows that $\lim_{t \to 0} |f_j(\gamma_{i,a}(t))| = \infty$ for each $i,j$, and therefore  
\begin{align*}
\lim_{t \to 0} \psi(\gamma_{i,a}(t)) &=  \lim_{t \to 0}([1:\gamma_{i,a}(t)], [1:f_1(\gamma_{i,a}(t))], [1:f_2(\gamma_{i,a}(t))]) \\
					 						&= ([0:0:1], [0:1], [0:1])\ . 
\end{align*}
Hence $([0:0:1], [0:1], [0:1])$ is in the closure of $H_i(a)$ in $\bar X$ for $i= 1,2$ and every $a \in Y$. It follows that $\psi$ does not preserve $(f_{\lambda^1}, f_{\lambda^2})$ at $\infty$ over {\em any} point in $Y$, as claimed.

Finally, we follow the proof of theorem \ref{filtrexistence-thm3} to find a completion which preserves map $f$ at $\infty$. The algebraic equations satisfied by $x_1$ and $x_2$ over $\cc[f_1,f_2]$ are:
\begin{gather*}
x_1^3 +  \frac{1}{1-c}(f_1 - f_2) x_1^2 +  \frac{1}{(1-c)^2}(f_1 - f_2)^2 x_1 -  \frac{1}{1-c}(f_1 - cf_2) = 0,\ \text{and} \\
x_2 - \frac{1}{1-c}(f_1 - f_2) = 0.
\end{gather*}
In the notations of the proof of theorem \ref{filtrexistence-thm3}, $d_{1,2}= d_{1,0} = 1$, $d_{1,1} = 2$ and $d_{2,0} = 1$. Let $\scrF := \{F_d : d \geq 0\}$ be the filtration defined by:
$$F_0 := \kk,\: F_1 := \kk\langle 1, x_1, x_2, f_1, f_2, x_1f_1, x_1f_2, x_1f_1^2, x_1f_1f_2, x_1f_2^2 \rangle,\: F_d := (F_1)^d\: \text{for}\ d \geq 2\: .$$
Construction of the proof of $\ref{filtrexistence-thm3}$ yields that $\xf$ preserves map $f$ at $\infty$. Indeed, $\psi_\scrF(x) = [1: x_1: x_2: f_1(x): f_2(x): x_1f_1(x): x_1f_2(x): x_1f_1^2(x): x_1f_1(x)f_2(x): x_1f_2^2(x)] \in \pp^9(\cc)$ for $x \in \affine{2}{\cc}$ and therefore applying \eqref{f-lambda-i} it follows that $\lim_{t \to 0} \psi_\scrF(\gamma_{i,a}(t)) =$
\begin{align*}
& \lim_{t \to 0} [1:\left(\frac{\lambda^i_1c + \lambda^i_2}{\lambda^i_1+\lambda^i_2} t +\littleoh(t^2) \right): \frac{1}{t} : \left( \frac{\lambda^i_2(1-c)}{\lambda^i_1+\lambda^i_2}\frac{1}{t} + \littleoh(t^2) \right): \left( \frac{\lambda^i_1(c-1)}{\lambda^i_1+\lambda^i_2}\frac{1}{t} + \littleoh(t^2) \right): \littleoh(1): \\ 
		&\phantom{\lim_{t \to 0} [\quad} \littleoh(1) : \left( \frac{(1-c)^2(\lambda^i_2)^2(\lambda^i_1c + \lambda^i_2)}{(\lambda^i_1+\lambda^i_2)^3}\frac{1}{t} + \littleoh(t^2) \right) : \left( -\frac{(1-c)^2\lambda^i_1\lambda^i_2(\lambda^i_1c + \lambda^i_2)}{(\lambda^i_1+\lambda^i_2)^3}\frac{1}{t} + \littleoh(t^2) \right): \\
		&\phantom{\lim_{t \to 0} [\quad} \left( \frac{(1-c)^2(\lambda^i_1)^2(\lambda^i_1c + \lambda^i_2)}{(\lambda^i_1+\lambda^i_2)^3}\frac{1}{t} + \littleoh(t^2)\right) ] \\
		&= [0 : 0 : 1 : \frac{\lambda^i_2(1-c)}{\lambda^i_1+\lambda^i_2} : \frac{\lambda^i_1(c-1)}{\lambda^i_1+\lambda^i_2} : 0 : 0: \frac{(1-c)^2(\lambda^i_2)^2(\lambda^i_1c + \lambda^i_2)}{(\lambda^i_1+\lambda^i_2)^3} : \\
		&\phantom{[\quad} -\frac{(1-c)^2\lambda^i_1\lambda^i_2(\lambda^i_1c + \lambda^i_2)}{(\lambda^i_1+\lambda^i_2)^3} :  \frac{(1-c)^2(\lambda^i_1)^2(\lambda^i_1c + \lambda^i_2)}{(\lambda^i_1+\lambda^i_2)^3}] \ .
\end{align*}
If $\lambda^1$ and $\lambda^2$ are linearly independent, then it follows that limits $\lim_{t \to 0} \psi_\scrF(\gamma_{i,a}(t))$ are different for $i=1,2$, and therefore, completion $\psi_\scrF$ preserves $\{f_{\lambda^1}, f_{\lambda^2}\}$ at $\infty$ over {\em every} $a \in Y = \affine{2}{\cc}$.
\end{example} 
\section{General Bezout-type theorems} \label{sec-bezout}
\subsection{Bezout theorem for Semidegrees} \label{sec-semi-bezout}
Let $X$ be an $n$ dimensional affine variety. Let $\delta$ be a complete degree like function on the coordinate ring $A$ of $X$ with associated filtration $\scrF := \{F_d : d \geq 0\}$. Recall from example \ref{proj-example4} that there exists $d > 0$ such that $(\adelta)^{[d]} := \dsum_{k \geq 0}F_{kd}$ is generated by $F_d$ as a $\kk$-algebra and then the $d$-uple embedding embeds $\xdelta$ into $\pp^l(\kk)$, where $l := \dim_\kk F_d - 1$.

Below we will make use of a notion of multiplicity at an isolated point $b$ of fiber $\finv(a)$ of morphism $f:X \to \ank$, where $X$ is an affine variety and $a := (a_1, \ldots, a_n) \in \ank$. The latter multiplicity we define following the definition in \cite[Example 12.4.8]{fultersection} as the intersection multiplicity at $b$ of the effective Cartier divisors determined by regular (on $X$) functions $f_j - a_j$, $1 \leq j \leq n$.

\begin{thm}[see {\cite[Theroem 1.3]{announcement}} and {\cite[Theorem 3.1.1]{preprint}}]\label{thm-affine-bezout}
Let $X$, $A$, $\delta$, $d$ and $l$ be as above. Denote by $D$ the degree of $\xdelta$ in $\pp^l(\kk)$. Let $f = (f_1, \ldots, f_n):X \to \ank$ be any generically finite map. Then for all $a \in \ank$, 
\begin{align}\label{semi-bezout}
|\finv(a)| \leq \frac{D}{d^n}\prod_{i=1}^n\delta(f_i) \tag{A}
\end{align}
where $|\finv(a)|$ is the number of the isolated points in fiber $\finv(a)$ each counted with the multiplicity of $\finv(a)$ at the respective point. If in addition $\delta$ is a semidegree and $\psi_\delta$ preserves $\{f_1, \ldots, f_n\}$ at $\infty$ over $a$, then \mathref{semi-bezout} holds with an equality.
\end{thm}

\begin{proof}
Let $a = (a_1, \ldots, a_n) \in \ank$. For each $i$, let $\aaa_i$ be the ideal generated by $(f_i - a_i)^d$ in $A$ and $\aaa_i^\delta := \bigoplus_{j \geq 0} (\langle (f_i - a_i )^d\rangle \cap  F_j)$ be the corresponding {\em homogeneous} ideal in $\adelta$. Clearly $G_i := ((f_i - a_i)^d)_{dd_i} \in \aaa_i$, where $d_i := \delta(f_i)$ for each $i = 1, \ldots, n$. It follows that
\begin{align}
\bigcap_{i=1}^n \{x \in X: (f_i(x) - a_i )^d = 0\} 	&\subseteq \bigcap_{i=1}^n \{x \in \xdelta:  G(x) = 0\ \text{for all}\ G \in \aaa_i^\delta\} \label{closure-f_i-d_i} \\
													&\subseteq \bigcap_{i=1}^n \{x \in \xdelta:  G_i(x) = 0\} \label{hypersurface-f_i-d_i}
\end{align}
where the first inclusion is due to lemma \ref{preservation-criterion}, since the closure of the hypersurface $V(\aaa_i)$ of $X$ in $\xdelta$ is $V(\aaa_i^\delta)$. Note that the sum of the multiplicities of the intersections of Cartier divisors determined by $(f_i - a_i)^d$ at the isolated points in the set on the left hand side of \eqref{closure-f_i-d_i} is precisely $d^n$ times the sum of the multiplicities of fiber $\finv(a)$ at the isolated points in $\finv(a)$. 

Pick a set of homogeneous coordinates $[y_0:\cdots: y_L]$ of $\pp^l(\kk)$. The $d$-uple embedding of $\xdelta$ into $\pp^l(\kk)$ induces a surjective homomorphism $\phi: \kk[y_0, \ldots, y_L] \to (\adelta)^{[d]}$ of graded $\kk$-algebras. For each $i$, choose an arbitrary homogeneous polynomial $\hat G_i \in \kk[y_0, \ldots, y_L]$ of degree $d_i$ such that $\phi(\hat G_i) = G_i$. According to the classical Bezout theorem on $\pp^l(\kk)$, the sum of the multiplicities of isolated points of $\xdelta \cap V(\hat G_1) \cap \cdots \cap V(\hat G_n)$ in $\pp^l(\kk)$ is at most $Dd_1 \cdots d_n$, and it is equal to $Dd_1 \cdots d_n$ if all the points in the intersection are isolated.

Since $\xdelta \cap V(\hat G_1) \cap \cdots \cap V(\hat G_n)$ is precisely $\xdelta \cap V(G_1) \cap \cdots \cap V(G_n)$, inequality \eqref{semi-bezout} follows from \eqref{hypersurface-f_i-d_i} and the conclusions of the two preceding paragraphs. The last assertion of theorem \ref{thm-affine-bezout} follows from the following observations: 

\begin{compactenum}
\item if $\delta$ is a semidegree, then according to lemma \ref{semi-principal} ideal $\aaa_i^\delta$ is generated by $G_i$ and hence $\subseteq$ in \eqref{hypersurface-f_i-d_i} can be replaced by $=$, and
\item completion $\psi_\delta$ preserves $\{f_1, \ldots, f_n\}$ at $\infty$ over $a$ iff the $\subseteq$ in \eqref{closure-f_i-d_i} is in fact an equality. \qedhere
\end{compactenum}
\end{proof}

\begin{rem}
Both assertions of theorem \ref{thm-affine-bezout} are valid for $\delta = \max_{j=1}^N \delta_j$ being a subdegree with $\delta_1(f_i) = \cdots = \delta_N(f_i)$ for all $i=1, \ldots, n$.
\end{rem}

\begin{rem}\label{semi-preserving-set}
Let $f := (f_1, \ldots, f_n):X \to \ank$ be a quasifinite map and $\delta$ be a complete semidegree on $A := \kk[X]$. We claim that if $\psi_\delta$ preserves $\{f_1, \ldots, f_n\}$ at $\infty$ over any point, then $\psi_\delta$ preserves $\{f_1, \ldots, f_n\}$ at $\infty$ over {\em all} points. Indeed, let $a := (a_1, \ldots, a_n) \in \ank$. For each $i$, let $\ppp_i(a)$ be the ideal generated by $(f_i- a_i)$ in $A$, and let $\ppp_i^\delta(a) := \bigoplus_{j \geq 0} (\ppp_i(a) \cap  F_j) \subseteq \adelta$. Due to lemma \ref{preservation-criterion}, $\psi_\delta$ preserves $\{f_1, \ldots, f_n\}$ at $\infty$ over $a$ iff $\sqrt {\scrI(a)} = \adelta_+$, where $\scrI(a) := \langle \ppp_1^\delta(a), \ldots, \ppp_n^\delta(a), (1)_1 \rangle \subseteq \adelta$. According to lemma \ref{semi-principal}, $\ppp_i^\delta(a) = \langle (f_i - a_i)_{d_i} \rangle$, where $d_i := \delta(f_i)$ for each $i$. Then $\scrI(a) = \langle (f_1 - a_1)_{d_1}, \ldots, (f_n - a_n)_{d_n}, (1)_1 \rangle = \langle (f_1)_{d_1}, \ldots, (f_n)_{d_n}, (1)_1 \rangle$. Since the latter expression is independent of $a_i$'s, the claim follows. Note that we have shown (in the notation of remark \ref{preserved-set}) that either $S_{\psi_\delta} = \emptyset$ or $S_{\psi_\delta} = f(X)$.
\end{rem}

\begin{example}[Weighted Bezout theorem, cf. \cite{damon}, {\cite[Example 7]{announcement}}, {\cite[Example 3.1.4]{preprint}}] \label{example-weighted-bezout}
Let $X := \ank$ and $\delta$ be a weighted degree on $A := \kk[x_1, \ldots, x_n]$ which assigns weights $d_i > 0$ to $x_i$, $1 \leq i \leq n$. Then $(1)_1, (x_1)_{d_1}, \ldots, (x_n)_{d_n}$ is a set of $\kk$-algebra generators of $\adelta$. A straightforward application of lemma \ref{preservation-criterion} implies that $\psi_\delta$ preserves at $\infty$ all components of the identity map $\identity$ of $\ank$ over $0$. Therefore theorem \ref{thm-affine-bezout} with $f = \identity$ and $d$ as in the preamble to theorem \ref{thm-affine-bezout} implies that $1 = \frac{D}{d^n}\prod_{i=1}^n d_i$ and therefore $D = \frac{d^n}{\prod_{i=1}^n d_i}$. Consequently, theorem \ref{thm-affine-bezout} implies for any $f$ that:
\begin{align}
|\finv(a)| &\leq \prod_{i=1}^n\frac{\delta(f_i)}{d_i}\ . \label{weighted-bezout}
\end{align}
Recall (example \ref{weighted-semi-example}) that $\gr \adelta \cong \kk[x_1, \ldots, x_n]$ via an identification of $[(g)_{\delta(g)}] \in \gr \adelta$ and $\ld_\delta(g) \in \kk[x_1, \ldots, x_n]$, for $0 \neq g \in A$. Let $a := (a_1, \ldots, a_n) \in \ank$ and $\aaa_i := \langle f_i - a_i \rangle $ (cf. the proof of theorem \ref{thm-affine-bezout}). According to lemma \ref{semi-principal}, ideals $\aaa_i^\delta$ are generated by $(f_i - a_i)_{\delta(f_i)} = (f_i)_{\delta(f_i)} - a_i((1)_1)^{\delta(f_i)}$ for $1 \leq i \leq n$. Therefore, 
\begin{align*}
\text{\eqref{weighted-bezout} holds with equality} & \iff \text{$\psi_\delta$ preserves $\{f_1, \ldots, f_n\}$ at $\infty$ over $a$} \\
												 & \iff V(\aaa_1^\delta, \ldots, \aaa_n^\delta, (1)_1) = \emptyset \subseteq \proj \adelta \\
												 & \iff V((f_1)_{\delta(f_1)}, \ldots, (f_n)_{\delta(f_n)}, (1)_1) = \emptyset \subseteq  \proj \adelta \\
												 & \iff V([(f_1)_{\delta(f_1)}], \ldots, [(f_n)_{\delta(f_n)}]) = \emptyset \subseteq  \proj \gr \adelta \\
												 & \iff V(\ld_\delta(f_1), \ldots, \ld_\delta(f_n)) = \{0\} \in \ank,
\end{align*}
where the last relation holds since the maximal ideal of the origin in $\ank$ corresponds to the irrelevant ideal $\dsum_{k > 0}F_k/F_{k-1}$ of $\gr\adelta$ via the isomorphism $\gr\adelta \cong \kk[x_1, \ldots, x_n]$. Finally note that formula \eqref{weighted-bezout} in combination with condition $V(\ld_\delta(f_1), \ldots, \ld_\delta(f_n)) = \{0\} \in \ank$ for equality in \eqref{weighted-bezout} constitute the content of the {\em Weighted Bezout theorem} stated in section \ref{subsec-background}. 
\end{example}

Applying theorem \ref{khothm}, number $D$ of theorem \ref{thm-affine-bezout} admits a geometric description as the volume of a convex body associated to $\delta$ and a $\zz^n$ valued (surjective) valuation $\nu$ of $\kk(X)$ (cf. section \ref{subsec-bezoutground}), namely:

\begin{prop}[see {\cite[Proposition 1.4]{announcement}} and {\cite[Theorem 3.1.3]{preprint}}]\label{propvanskii}
Let $X$, $A$, $d$, $\delta$, $l$ and $D$ be as in theorem \ref{thm-affine-bezout} and $\nu$ be a valuation on $A$ with values in $\zz^n$.
Let $C$ be the smallest closed cone in $\rr^{n+1}$ containing
$$G := \{(\frac{1}{d}\delta(f),\nu(f)): f \in A\} \cup \{(1,0, \ldots, 0)\}.$$ 
Let $\Delta$ be the convex hull of the cross-section of $C$ with the first coordinate having value $1$. Then $D = n!\vol_n(\Delta)$, where $\vol_n$ is the $n$-dimensional Euclidean volume.
\end{prop}

\begin{proof}
Consider $L := F_d = \{f \in A: \delta(f) \leq d\}$. As in section \ref{subsec-bezoutground}, we introduce $S(L) := \{(k, \nu(f)): f \in L^k,\ k \geq 0\} \subseteq \nn \dsum \zz^n$. Let $C(L)$ be the closure of the convex hull of $S(L)$ in $\rr^{n+1}$ and $\Delta(L) := C(L) \cap (\{1\} \times \rr^n)$. Mapping $\Phi_L$ of theorem \ref{khothm} is precisely the embedding $X \subseteq \xdelta \into \pp^l(\kk)$, so that the mapping degree $d(L)$ of $\Phi_L$ is $1$. Therefore theorem \ref{khothm} implies that $[L, \ldots, L] = \frac{n!}{s(L)}\vol_n(\Delta(L))$, where $s(L)$ is the index in $\zz^n$ of the subgroup $S'(L)$ generated by all the differences $\alpha - \beta$ such that $(k, \alpha), (k, \beta) \in S(L)$ for some $k \geq 1$. Note also that:
\begin{compactenum}
\item \label{D-is-LLLL} according to its definition $[L,\ldots, L]$ is equal to the degree $D$ of $\xdelta$ in $\pp^L(\cc)$,
\item \label{F_d-generates-all} $F_d$ generates $(\adelta)^{[d]}$, so that for $L^k = (F_d)^k = F_{kd}$ for all $k \geq 1$, and therefore for all $f \in A$, $(k, \nu(f)) \in S(L)$ for all $k \geq \frac{\delta(f)}{d}$, and
\item \label{1-in-all} for all $k \geq 1$, $1 \in L^k$ and therefore $(k, \nu(1)) =  (k,0,\ldots, 0) \in S(L)$. 
\end{compactenum}
Properties \ref{F_d-generates-all} and \ref{1-in-all} imply that $S'(L) \supseteq \{\nu(f): f \in A\}$. Since $\nu$ is surjective, it follows that $S'(L) = \zz^n$ and hence $s(L) = 1$. But then $D = n!\vol_n(\Delta(L))$ due to property \ref{D-is-LLLL}. Finally, properties \ref{F_d-generates-all} and \ref{1-in-all} also imply that $C = C(L)$, and consequently that $\Delta = \Delta(L)$.
\end{proof}

%

\begin{example}
Let $X = \ank$ and $\delta$ be as in example \ref{example-weighted-bezout}. Let $\nu$ be any monomial valuation on $A :=\kk[x_1, \ldots, x_n]$, e.g. the one that assigns to $\sum a_\alpha x^\alpha \in A\setminus \{0\}$ the lexicographically (coordinatewise) minimal exponent $\alpha$ among all $\alpha := (\alpha_1, \ldots, \alpha_n)$ with $a_\alpha \neq 0$. With $d$ being any common multiple of $d_1, \ldots, d_n$, it is straightforward to see that $\Delta = \{(1,x) \in \rr_+^{n+1}: \sum_{i=1}^n x_id_i \leq d\}$ and therefore $\vol_n(\Delta) = \frac{1}{n!}\prod_{i=1}^n\frac{d}{d_i}$. It follows that $D = n!\vol_n(\Delta) = \prod_{i=1}^n\frac{d}{d_i}$, which is, of course, the value we have calculated in example \ref{example-weighted-bezout}.
\end{example}
\subsection{Bezout theorem for Subdegrees}\label{sec-sub-bezout}
\newcommand{\adeltaa}[1]{A^{\delta_{#1}}}
\newcommand{\adeltaaa}[2]{(A^{\delta_{#1}})^{[#2]}}
\newcommand{\xdeltaa}[1]{X^{\delta_{#1}}}

\newcommand{\divclosure}[1][]{%
  	\ensuremath{{\overline \Div}^{#1}}%
}
\newcommand{\divclosuredelta}[1][]{%
  	\divclosure[\delta_{#1}]%
}
\newcommand{\divdelta}[1][]{%
  	\Div_{X^{\delta_{#1}}}%
}
\newcommand\divinfinitydeltaa[2]{%
	D^{\delta_{#1}}_{d_{#2}, \infty}%
}
\newcommand\divinfinitydelta[1]{%
	\divinfinitydeltaa{#1}{#1}%
}

Assume $f := (f_1, \ldots, f_n):X \to \ank$ is a dominating morphism of affine varieties with generically finite fibers and $\delta := \max\{\delta_j: 1 \leq j \leq N\}$ is a \complete subdegree on $A := \kk[X]$, i.e. $\delta$ is non-negative, finitely generated and $\delta^{-1}(0) = \kk \setminus \{0\}$. Assume $\delta_j(f_i) > 0$ for each $i,j$. In the spirit of theorem \ref{thm-affine-bezout} we derive in this section an upper bound for the number of points in a generic fiber of $f$ in $X$ (counted with multiplicity) in terms of the degree of a projective completion of $X$.

\begin{defn*}
\mbox{}
\begin{compactitem}
\item Let $g \in A$ and $\Div_X(g)$ be the principal Cartier divisor corresponding to $g$ on $X$. Assume that the corresponding Weil divisor is $[\Div_X(g)] = \sum r_i[V_i]$. Given a completion $X \into Y$ of $X$, we write $[\divclosure[Y]_X(g)]$ for the Weil divisor on $Y$ given by:
$$[\divclosure[Y]_X(g)] := \sum r_i [\overline V_i],$$
where $\overline V_i$ is the closure of $V_i$ on $Y$. If $Y = \xdelta$ for some degree like function $\delta$ on $A$, then we also make use of notation $\divclosuredelta_X(g)$ for $\divclosure[Y]_X(g)$.
\item If $g \in A$ is such that $\delta_j(g) > 0$ for all $j= 1, \ldots, N$, then 
$$\delta_g := e_g\cdot\max\{\frac{\delta_j}{\delta_j(g)}: 1 \leq j \leq N\},$$ 
where $e_g$ is a suitable integer to ensure that $\delta_g$ is integer valued (e.g. one can take $e_g := \prod_{j=1}^N \delta_j(g)$).
\end{compactitem}
\end{defn*}

Let the filtration corresponding to $\delta$ be $\scrF := \{F_d: d \geq 0\}$. Identify $\adelta$ with $\sum_{d \in \zz}F_dt^d$. Recall that $\xdelta := \proj \adelta$ is the union of affine charts of the form $\spec \adelta_{(gt^d)}$, where $d > 0$ and $\adelta_{(gt^d)}$ is the subring of elements of degree zero of the localizations $\adelta_{gt^d}$. Say $U_0, \ldots, U_m$ is an open cover of $\xdelta$ with $U_j := \spec \adelta_{(g_jt^{l_j})}$ for some $l_j \geq 1$ and $g \in F_{l_j}$ for every $j$. Moreover, assume that $g_0 = 1$ and $l_0 = 1$, so that $U_0 = \spec \adelta_{(t)} = \spec A$. Let $d$ be a common multiple of $l_1, \ldots, l_m$. Then for each $j$, $h_j := \frac{t^d}{(g_j)^{d/l_j}t^d}$ is a regular function on $U_j$ and $h_j/h_k$ is a unit on $U_j \cap U_k$. Therefore collection $\{(h_j,U_j)\}_j$ defines an {\em effective} Cartier divisor $D^\delta_{d, \infty}$ on $\xdelta$, which we call the {\em $d$-uple divisor at infinity}. Its associated Weil divisor is
$$[D^\delta_{d, \infty}] := \sum_{j=1}^N \ord_j(D^\delta_{d, \infty})[V_j],$$
where $V_1, \ldots, V_N$ are the irreducible components of $X_\infty$ and $\ord_j$ is the shorthand for $\ord_{V_j}$ (where $\ord_{V_j}$ is as defined in section \ref{divisor-subsection}). Support of $[D^\delta_{d, \infty}]$ being $X_\infty$ justifies index $\infty$ as a subscript of $D^\delta_{d, \infty}$. 

\begin{lemma} \label{divisor-at-infinity}
Let $X$, $A$, $\delta$ and $D^\delta_{d, \infty}$ be as above. Then
\begin{compactenum}
\item \label{div-at-infinity} $[D^\delta_{d, \infty}] = \sum_{j=1}^N \frac{d}{d_j}[V_j]$ where for every $j$, integer $d_j$ is the positive generator of the subgroup of $\zz$ generated by $\{\delta_j(f): f \in A\}$. 
\item \label{div-g-d-delta-g-d} Let $g \in A$ be such that $\delta_g$ is finitely generated. Then the principal divisor of $g^{d}$ on $X^{\delta_g}$ is $[\divdelta[g](g^{d})] = d[\divclosuredelta[g]_X(g)] - e_g[\divinfinitydeltaa{g}{}]$.
\end{compactenum}
\end{lemma}

\begin{proof}
\begin{asparaenum}
\item Fix integer $j$, $1 \leq j \leq N$. Local ring $\localj$ is a discrete valuation ring and its associated valuation is $\nu_j(\cdot) := -\frac{\delta_j(\cdot)}{d_j}$ (proposition \ref{pole-and-degree}). Pick $k$, $1 \leq k \leq N$, such that $V_j \cap U_k \neq \emptyset$. Recall that $U_k := \spec \adelta_{(g_kt^{l_k})}$ and a local equation for $D^\delta_{d, \infty}$ on $U_k$ is $\frac{t^d}{(g_k)^{d/l_k}t^d}$. Let $\ppp_j$ be the ideal of $\adelta$ corresponding to $V_j$. Since $V_j \cap U_k \neq \emptyset$, it follows that $g_kt^{l_k} \not\in \ppp_j$ and therefore $\delta_j(g_k) = l_k$ according to assertion \ref{part1-gensgf-characterize-lemma} of lemma \ref{gensgf-characterize-lemma}. Therefore $\ord_j(D^\delta_{d, \infty}) = \nu_j(\frac{t^d}{(g_k)^{d/l_k}t^d}) = \nu_j(1/g_k^{d/l_k}) = -\frac{d}{l_k}\nu_j(g_k) = \frac{d}{l_k}\cdot\frac{l_k}{d_j} = \frac{d}{d_j}$. It follows that $[D^\delta_{d, \infty}] := \sum_{j=1}^N \ord_j(D^\delta_{d, \infty})[V_j] = \sum_{j=1}^N \frac{d}{d_j}[V_j]$, which completes the proof of assertion \ref{div-at-infinity}.
\item Reindexing the $\delta_j$'s if necessary, we may assume that the minimal presentation of $\delta_g$ is $\delta_g = \max\{\frac{e_g \cdot \delta_j}{\delta_j(g)}: 1 \leq j \leq M\}$ for some $M \leq N$. For $1 \leq j \leq M$, let $V'_j$ be the irreducible component of $X^{\delta_g}_\infty$ corresponding to $\delta_j$ and $d'_j$ be the positive generator of the subgroup of $\zz$ generated by $\{\frac{e_g \cdot \delta_j(h)}{\delta_j(g)}:h \in A\}$. Then, as a straightforward consequence of proposition \ref{pole-and-degree} it follows that $\ord_{V'_j}(g^d) = -  \frac{e_gd}{d'_j}$ for every $j$ and therefore $[\divdelta[g](g^d)] = d[\divclosuredelta[g]_X(g)] + \sum_{j=1}^M \ord_{V'_j}(g^d)[V'_j] = d[\divclosuredelta[g]_X(g)] - \sum_{j=1}^M \frac{e_gd}{d'_j}[V'_j]$. On the other hand, applying assertion \ref{div-at-infinity} with $\delta_g$ in place of $\delta$ yields that $[D^{\delta_g}_{d, \infty}] = \sum_{j=1}^M \frac{d}{d'_j}[V'_j]$. Therefore $[\divdelta[g](g^{d})] = d[\divclosuredelta[g]_X(g)] - e_g[\divinfinitydeltaa{g}{}]$, as required. \qedhere
\end{asparaenum}
\end{proof}

\begin{thm}\label{quasi-affine-bezout-prelim}
Assume that $f := (f_1, \ldots, f_n):X \to \ank$ and $\delta := \max\{\delta_j: 1 \leq j \leq N\}$ are as in the first paragraph of this section and that $\delta_j(f_i) > 0$, $1 \leq i, j \leq n$. Assume also that for $i=1, \ldots, n$, subdegrees $\delta_{f_i}$ are finitely generated. Let $d_{f_i} \geq 1$, $1 \leq i \leq n$, be such that the $d_{f_i}$-uple embedding of $X^{\delta_{f_i}}$ is a closed immersion of $X^{\delta_{f_i}}$ into a projective space $\pp^{L_i}(\kk)$. Let $L := \prod_{i=1}^n (L_i + 1) - 1$ and $\bar X$ be the closure of the image of $X$ in $\pp^L(\kk)$ under the composition of the following maps:
$$X \into X^{\delta_{f_1}} \times \cdots \times X^{\delta_{f_n}} \into \pp^{L_1}(\kk) \times \cdots \times \pp^{L_n}(\kk) \into \pp^{L}(\kk),$$
where the first map is the diagonal embedding and the last map is the Segre embedding. Then for all $a \in \ank$,
\begin{align} \label{quasi-bezout}
|\finv(a)| &\leq \frac{e_{f_1}\cdots e_{f_n}}{n^n d_{f_1}\cdots d_{f_n}} \deg(\bar X), \tag{C}
\end{align}
where $|\finv(a)|$ is the number of the isolated points in $\finv(a)$ each counted with the multiplicity of $\finv(a)$ at the respective point. 
\end{thm}

\begin{description}
\item[Question:] Is it true that the completeness of $\delta$ implies the finite generation property of every $\delta_{f_i}$?
\end{description}

\begin{proof}[Proof of theorem \ref{quasi-affine-bezout-prelim}]
Denote by $[y_{i,0}:\cdots:y_{i,L_i}]$ the homogeneous coordinates on $\pp^{L_i}(\kk)$. Without loss of generality we may assume that $X^{\delta_{f_i}}_\infty = X^{\delta_{f_i}} \cap V(y_{i,0})$. 

Denote by $X^\eta$ the closure of the diagonal embedding of variety $X$ into the product $X^{\delta_{f_1}} \times \cdots \times X^{\delta_{f_n}} \subseteq \pp^{L_1}(\kk) \times \cdots \times \pp^{L_n}(\kk) =: Y$. Denote by $[y_0: \cdots: y_L:z_1: \cdots: z_n]$ the homogeneous coordinates on $\pp^{L'}(\kk)$, where $L' := L + n$. Let us identify $\pp^L(\kk)$ with the subspace $V(z_1, \cdots, z_n)$ of $\pp^{L'}(\kk)$ and let $s: \pp^{L_1}(\kk) \times \cdots \times \pp^{L_n}(\kk) \to \pp^L(\kk)$ denote the Segre embedding. Let $s': X^\eta \to \pp^{L'}(\kk)$ be the map defined by:
$$s' :~ X^\eta \ni ([y_{1,0}:\cdots:y_{1,L_1}], \ldots, [y_{n,0}:\cdots:y_{n,L_n}]) \mapsto [s(y):(y_{1,0})^n : \cdots: (y_{n,0})^n] \in \pp^{L'}(\kk). $$

Fix an $i$, $1 \leq i \leq n$. Let $D_i := \pi_i^*(D^{\delta_{f_i}}_{d_{f_i}, \infty})$, where $\pi_i:X^\eta \to \xdeltaa{f_i}$ is the projection onto the $i$-th factor of $Y$. Due to our choice of $y_{i,0}$, Cartier divisor $D^{\delta_{f_i}}_{d_{f_i}, \infty}$ is precisely the restriction of the divisor of $y_{i,0}$ to $X^{\delta_{f_i}}$. It follows that ${s'}^*(D'_i) = nD_i$, where $D'_i$ is the restriction of the divisor of $z_i$ to $s'(X^\eta)$. 

Then $(D_1, \ldots, D_n) = \frac{1}{n^n}({s'}^*(D'_1), \ldots, {s'}^*(D'_n)) = \frac{1}{n^n}(D'_1, \ldots, D'_n)$, since intersection numbers are preserved under the pull backs by proper birational morphisms \cite[Example 2.4.3]{fultersection}. Since each $D'_i$ is the divisor of a linear form on $s'(X^\eta)$, it follows that the intersection number $(D'_1, \ldots, D'_n) = \deg s'(X^\eta)$. On the other hand, $s|_{X^\eta} = \pi' \circ s'$, where $\pi'$ is the projection onto the first $L$ coordinates. Since the mapping degree of $\pi'|_{s'(X^\eta)}$ is $1$, it follows that $\deg (s(X^\eta)) = \deg s'(X^\eta)$ \cite[Proposition 5.5]{mummetry}. Combining three equalities established in this paragraph it follows that $(D_1, \ldots, D_n) = \frac{1}{n^n}\deg s(X^\eta)$.

Let $E_i := \pi_i^*(\divclosuredelta[f_i]_X(f_i))$. According to assertion \ref{div-g-d-delta-g-d} of lemma \ref{divisor-at-infinity} $[\divdelta[f_i](f_i^{d_{f_i}})] = d_{f_i}[\divclosuredelta[f_i]_X(f_i)] -e_{f_i}[\divinfinitydelta{f_i}]$. Since $\pi_i^*(\divdelta[f_i](f_i)) = \Div_{X^\eta}(f_i)$ it follows that \begin{align*}
(E_1, \ldots, E_n) = \frac{e_{f_1}\cdots e_{f_n}}{d_{f_1}\cdots d_{f_n}}(D_1, \ldots, D_n) = \frac{e_{f_1}\cdots e_{f_n}}{n^n d_{f_1}\cdots d_{f_n}}\deg s(X^\eta)\ .
\end{align*}

Fix an $i$, $1 \leq i \leq n$. Note that $\divinfinitydelta{f_i}$ are {\em very ample}, i.e. are the pull backs of the hyperplane sections under the embeddings of $X^{\delta_{f_i}}$ into $\pp^{L_i}(\kk)$. In particular, these divisors are {\em base point free} \cite[Section II.7]{hart}. Then the pull backs $D_i$ of $\divinfinitydelta{f_i}$ are also base point free, and so are $e_{f_i}E_i$ (the latter being linearly equivalent to $d_{f_i}D_i$). Also, since $E_i$'s are {\em effective} (defined in section \ref{divisor-subsection}), it follows that the intersection number $(E_1, \ldots, E_n)$ bounds the sum of the intersection multiplicities of $E_i$'s at the isolated points of the intersection $\bigcap_{i=1}^n \supp(E_i)$ \cite[Section 12.2]{fultersection}. Of course $X \cap (\bigcap_{i=1}^n \supp(E_i)) = \finv (0)$. Therefore $|\finv(0)| \leq (E_1, \ldots, E_n)$, which completes the proof of the theorem.
\end{proof}

\noindent {\bf Future plans:} 
\begin{asparaenum}
\item Let $f := (f_1, \ldots, f_n):X \to \ank$ be any generically finite map (not necessarily satisfying the hypotheses of theorem \ref{quasi-affine-bezout-prelim}). Replacing $f$ by $\xi \circ f$ for a generic affine transformation $\xi$ of $\ank$ (and reordering $\delta_j$'s if necessary), one may assume that there is an $M \leq N$ such that 
\begin{compactenum}
\item $\delta_j(f_i) > 0$ for all $i$ and all $j = 1, \ldots, M$, and 
\item $\delta_j(f_i) = 0$ for all $i$ and all $j = M+1, \ldots, N$.
\end{compactenum}
We expect that it should be possible to extend theorem \ref{quasi-affine-bezout-prelim} in this setting as a consequence of extending the arguments of the proof of theorem \ref{quasi-affine-bezout-prelim} to the case of $\tilde X := \spec \tilde A$, where $\tilde A := \{g \in A: \delta_j(g) \leq 0$ for $M+1 \leq j \leq N\}$, and for the completion ${\tilde X}^{\tilde \delta}$ of $\tilde X$ determined by $\tilde \delta := \max\{\delta_j: 1 \leq j \leq M\}$. 
\item Moreover, we hope to prove (by means of an extension of our theorem \ref{bern-condition-at-infinity} to the case of subdegrees) that for completion $\xdelta$ that preserves map $f$ at $\infty$ (in the setting of theorem \ref{quasi-affine-bezout-prelim}), inequality in \eqref{quasi-bezout} can be replaced by equality.
\end{asparaenum}
\section{Iterated Semidegrees}\label{sec-iterated}
\newcommand{\atildedelta}{\profinggg{\tilde \delta}}
\newcommand{\xtildedelta}{X^{\tilde \delta}}

In this section we describe particularly simple semidegrees generalizing weighted homogeneous degrees for which we establish a constructive version of affine Bezout-type theorem. Our dream is that a stronger version of Main Existence Theorem \ref{thm-noeth-integral-subdegree} would be valid with subdegrees in whose minimal presentations only constructive semidegrees `like' the iterated semidegrees of this section would appear, and we expect that a precise constructive version of an affine Bezout-type theorem for any generically finite map $f: \ank \to \ank$ would follow.

Let $A$ be a domain and $\delta$ be a degree like function on $A$. Pick $f \in A$ and an integer $w$ with $w < \delta(f)$. Let $s$ be an indeterminate over $A$ and $\delta_e$ be a `natural' extension of $\delta$ to $A[s]$ such that $\delta_e(s) = w$, namely: $\delta_e(\sum a_is^i) := \max_{a_i \neq 0}(\delta(a_i)+iw)$. Of course $\delta$ is a degree like function iff $\delta_e$ is a degree like function.

\begin{lemma} \label{semi-extension}
$\delta$ is a semidegree iff $\delta_e$ is a semidegree.
\end{lemma}

\begin{proof}
The $(\Leftarrow)$ direction is obvious, since $\delta \equiv \delta_e|_A$. For the proof of the `only if' implication, let $G := \sum g_is^i, H := \sum h_js^j \in A[s]$ with $d := \delta_e(G), e := \delta_e(H)$. For each $k \geq 0$, let $G_k := \sum_{\delta(g_i)+ iw = k} g_is^i$ and $H_k := \sum_{\delta(h_i)+ iw = k} h_is^i$. It suffices to show that $\delta_e(G_dH_e) = d + e$, and hence we may w.l.o.g. assume $G = G_d$ and $H = H_e$. Let $i_0$ (resp. $j_0$) be the largest integer such that $g_{i_0} \neq 0$ (resp. $h_{j_0} \neq 0$). Then $GH = g_{i_0}h_{j_0}s^{i_0+j_0} + \sum_{m<i_0+j_0} a_ms^m$. Thus $\delta_e(GH) \geq \delta_e(g_{i_0}h_{j_0}s^{i_0+j_0}) = \delta(g_{i_0}h_{j_0})+ (i_0+j_0)w = \delta(g_{i_0})+ i_0w + \delta(h_{j_0})+ j_0w = d+e$. Since the inequality $\delta_e(GH) \leq d + e$ is obviously true, it follows that $\delta_e(GH) = d + e$. 
\end{proof}

\begin{rem} \label{sub-extension}
In fact $\delta$ is a subdegree iff $\delta_e$ is a subdegree, which follows by means of calculations similar to those in the proof of lemma \ref{semi-extension} and of the characterization of subdegrees as degree like functions $\eta$ satisfying $\eta(f^k) = k\eta(f)$ for all $f \in A$ and $k \geq 0$ (corollary \ref{sgf-reformulation}).
\end{rem}

Let $J$ denote the ideal generated by $ s - f$ in $A[s]$. Identify $A$ with $A[s]/J$ and define $\tilde\delta$ to be the degree like function on $A$ {\em induced by} $\delta_e$, i.e. $\tilde \delta(g) := \min\{\delta_e(G): G - g \in J\}$. Let $\aaa$ be the principal ideal generated by $f$ in $A$ and let $\aaa^\delta$ be the ideal induced in $\adelta$ by $\aaa$ (as defined in lemma \ref{preservation-criterion}). Denote by $\gr \aaa$ the ideal generated by the image of $\aaa^\delta$ in $\gr \adelta$, i.e. $\gr \aaa := \langle [(g)_{\delta(g)}]: g \in \aaa \rangle$, where $[(g)_{\delta(g)}]$ denotes the equivalence class of $(g)_{\delta(g)}$ in $\gr \adelta$.

\begin{rem} \label{gr-a-unit-bad}
A straightforward application of definitions shows that $\tilde \delta$ is a degree like function provided that $\delta_e$ is a degree like function and $\tilde \delta \equiv 0$ on $\kk$. Note that $\tilde \delta$ is a meaningful degree like function only if $[(f)_{\delta(f)}]$ is not a unit in $\gr \adelta$. Indeed, if $[(f)_{\delta(f)}]$ is a unit in $\gr \adelta$, then $[(1)_0] =  [(f)_{d}][(g)_{-d}] \in \gr \adelta$ for some $g \in A$ with $\delta(g) = -d$. It follows that if $1 - fg \neq 0$ then $\delta(1-fg) < 0$. Let $G := 1 - fg + gs \in A[s]$. Then $\delta_e(gs) = w - d < 0$ and therefore $\delta_e(G) < 0$. Moreover, $1 \equiv G \mod J$ in $A[s]$. Consequently $\tilde \delta(1) \leq \delta_e(G) < 0$. Since, for all $h \in A$ and $n \in \zz_+$, $\tilde\delta(h) = \tilde \delta(h\cdot (1)^n) \leq \tilde \delta(h) + n \tilde\delta(1)$, it follows that $\tilde\delta(h) = - \infty$ for all $h \in A$. Moreover, $(1)_1$ is a unit in $\profinggg{\tilde \delta}$ (since $((1)_1)^{-1} = (1)_{-1}$) and therefore $\gr \profinggg{\tilde \delta} \cong \profinggg{\tilde \delta}/ \langle (1)_1 \rangle$ is the zero ring.
\end{rem}

\begin{thm}[cf. {\cite[Example 5]{announcement}} and {\cite[Theorem 2.1.3]{preprint}}]\label{iterated-thm} 
\mbox{}
\begin{compactenum}
\labelformat{enumi}{\theenumi.}
\labelformat{enumii}{\theenumi(\theenumii)}
\item \label{trivial-assertions}
\begin{compactenum}[(a)] 
\item \label{nonnegative-iteration} If $\delta$ is non-negative and $w > 0$, then $\tilde \delta$ is non-negative.
\item \label{iterated-equality} $\tilde \delta(g) \leq \delta(g)$ for all $g \in A$. If $[(g)_{\delta(g)}] \not\in \gr \aaa$, then $\tilde \delta(g) = \delta(g)$.
\item \label{iterated-inequality} $\tilde \delta(f) \leq w < \delta(f)$. If $\delta$ is a semidegree and $[(f)_{\delta(f)}]$ is not a unit in $\gr \adelta$, then $\tilde \delta(f) = w$.
\end{compactenum}
\item \label{semi-iteration} Assume $\delta$ is a semidegree. Then 
\begin{compactenum}[(a)]
\item \label{total-isomorphism} The homomorphism of $\kk$-algebras $A[s]^{\delta_e} \cong \adelta[s] \to \profinggg{\tilde \delta}$ induced by the inclusion $\adelta \into \profinggg{\tilde \delta}$ (available due to $\tilde \delta \leq \delta$) and $s \mapsto (f)_{w}$ is surjective with kernel $\langle (s - f)_{\delta_e(s - f)} \rangle$.
\item \label{graded-isomorphism} $\gr \profinggg{\tilde \delta}$ is isomorphic to $(\gr \adelta/\gr a)[z]$, where $z$ is an indeterminate of degree $w$ and the isomorphism maps $z$ to the class of equivalence $[(f)_w]$ of $(f)_w \in \profinggg{\tilde \delta}$.
\item \label{can-iterate-semidegree} $\tilde\delta$ is a semidegree if and only if $\gr \aaa$ is a prime ideal of $\gring$.
\end{compactenum}
\end{compactenum}
\end{thm}

Note that if $\delta$ is a semidegree, then $\gr \aaa$ is a principal ideal in $\gr \adelta$ generated by $[(f)_{\delta(f)}]$ due to lemma \ref{semi-principal}.

\begin{proof}
Assertion \ref{nonnegative-iteration} is a straightforward consequence of the definitions. Note that $\tilde \delta(g) \leq \delta_e(G)$ for all $g \in A$ and $G \in A[s]$ such that $G \equiv g\mod J$. The first assertion of \ref{iterated-equality} follows by setting $G := g$ in the previous sentence. Similarly, the first assertion of \ref{iterated-inequality} follows by setting $g := f$ and $G := s$. As for the second assertion of \ref{iterated-equality}, let $g \in A$ and $G \in A[s]$ be such that $G \equiv g \mod J$. Let $a_0, \ldots, a_k \in A$ such that $G = g + (s - f)(a_0 + a_1s + \cdots + a_ks^k) = (g - fa_0) + \sum_{i=1}^k (a_{i-1} - fa_i)s^i + a_ks^{k+1}$. Note that if $e := \delta(g - fa_0) < d := \delta(g)$, then $\delta(fa_0) = d$ and $(g)_{d} = (fa_0)_{d} + ((1)_1)^{d-e}(g-fa_0)_e$, so that $[(g)_{d}] = [(fa_0)_{d}] \in \gr \aaa$. In other words, if $[(g)_{\delta(g)}] \not\in \gr \aaa$, then $\delta(g-fa_0) \geq \delta(g)$ and therefore $\delta_e(G) \geq \delta(g)$. It follows that $\tilde \delta(g) \geq \delta(g)$, which concludes the proof of \ref{iterated-equality}.

Next we prove the second assertion of \ref{iterated-inequality}. Assume $\delta$ is a semidegree and contrary to the conclusion of \ref{iterated-inequality} that $\tilde \delta(f) < w$. Then it suffices to show that $[(f)_{d}]$ is a unit in $\gr \adelta$, where $d := \delta(f)$. Indeed, $\tilde \delta(f) < w$ in view of the definition of $\tilde \delta$ in terms of $\delta_e$ implies that there is an identity
\begin{align} \label{tilde-identity}
f = a_kf^k + a_{k+1}f^{k+1} + \cdots + a_lf^l
\end{align}
with $a_k, \ldots, a_l \in A$ such that for all $j$, $0 \leq k \leq j \leq l$, $\delta(a_j) + jw < w$. In particular, $\delta(a_0) < w$ if $a_0 \neq 0$ and $\delta(a_1) < 0$ if $a_1 \neq 0$. 

If $k > 1$, then dividing both sides of \eqref{tilde-identity} by $f$ it follows that $fg = 1$, where $g := \sum_{j=k}^l a_jf^{j-2}$. Then $\delta(g) = -\delta(f) = -d$ and therefore $[(f)_d]\cdot [(g)_{-d}] = [(1)_0] \in \gr\adelta$. Since $[(1)_0]$ is the identity in $\gr \adelta$, it follows that $[(f)_d]$ is a unit in $\gr \adelta$, which proves assertion \ref{iterated-inequality} in the case that $k > 1$.  

If $k = 1$, then \eqref{tilde-identity} implies that $1 - a_1 = fg_1$, where $g_1 := a_2 + a_3f + \cdots + a_lf^{l-2}$. Since $\delta(a_1) < 0$, it follows that $\delta(1 - a_1) = 0$ and therefore $\delta(g_1) = -d$. Moreover, $[(a_1)_0] = 0 \in \gr \adelta$. Hence $[(f)_d]\cdot [(g_1)_{-d}] = [(1)_0] \in \gr \adelta$. Consequently $[(f)_d]$ is a unit in $\gr \adelta$, as required.

It remains to consider the case of $k=0$. In this case $a_0 = fg_2$, with element $g_2 :=$ $1 - a_2f - a_3f^2 - \cdots - a_lf^{l-1}$. Then $\delta(g_2) = \delta(a_0) - \delta(f) < w - d < 0$ and $g_2 = 1 - fg_1$ with $g_1 := a_2 + a_3f + \cdots + a_lf^{l-2}$. Since $\delta(g_2) < 0 = \delta(1)$, it follows that $\delta(fg_1) = \delta(1) = 0$, and therefore $\delta(g_1) = - \delta(f) = -d$. Consequently $[(f)_d]\cdot [(g_1)_{-d}] = [(1)_0] - [(g_2)_0] = [(1)_0] \in \gr\adelta$, implying $[(f)_d]$ is a unit in $\gr \adelta$, which completes the proof of \ref{iterated-inequality}.

Next we prove assertion \ref{semi-iteration} Assume $\delta$ is a semidegree. Due to remark \ref{gr-a-unit-bad} it suffices to consider the case when $[(f)_{\delta(f)}]$ is not a unit in $\gr \adelta$. Then $\delta(f) = w$ according to assertion \ref{iterated-inequality}. We start with introducing two surjective $\kk$-algebra homomorphisms $\phi: A[s] \onto A$ and $\Phi: \profingg{A[s]}{\delta_e} \onto \profinggg{\tilde \delta}$ by means of formulae
\begin{align*}
\phi(\sum_{i=1}^k a_is^i) &:= \sum_{i=1}^k a_if^i\quad \text{for any}\ a_1, \ldots, a_k \in A, \\
\Phi((H)_{d}) &:= (\phi(H))_d\quad \text{for all}\ H \in A[s],\ d \geq \delta_e(H) \in \zz.
\end{align*}
Clearly $\phi$ is surjective and $\ker \phi = J$. It follows that $\Phi$ is a surjective homomorphism of graded rings with $\ker \Phi = J^{\delta_e}$, and consequently, $\profinggg{\tilde \delta} = \profingg{A[s]}{\delta_e}/J^{\delta_e}$. Moreover, $\delta$ being a semidegree on $A$ implies that $\delta_e$ is also a semidegree on $A[s]$ (lemma \ref{semi-extension}) and therefore $J^{\delta_e} = \langle (s - f)_{\delta_e(s - f)} \rangle$ (lemma \ref{semi-principal}), which completes the proof of \ref{total-isomorphism}.

Next we prove assertion \ref{graded-isomorphism}. Ring $\gr \profinggg{\tilde \delta} = \profinggg{\tilde \delta}/\langle (1)_1 \rangle \cong \profingg{A[s]}{\delta_e}/(J^{\delta_e} + \langle (1)_1 \rangle)$ because $\profinggg{\tilde \delta}$ $=$ $\profingg{A[s]}{\delta_e}/J^{\delta_e}$. The element $z$ in the assertion of \ref{graded-isomorphism} (which corresponds to $[(f)_w]$) is precisely the equivalence class $[(s)_w]$ of $(s)_w \in \profingg{A[s]}{\delta_e}$. Note that the homomorphism defined by $\adelta[s] \ni \sum (f_i)_{d-iw}s^i \mapsto (\sum f_is^i)_d \in \profingg{A[s]}{\delta_e}$ is an isomorphism. Since $\delta_e(s) = w < \delta(f) = \delta_e(f)$, it follows that $(s - f)_{\delta_e(s - f)} = (s)_w(1)_{\delta(f) - w} + (f)_{\delta(f)}$ and therefore $J^{\delta_e} + \langle (1)_1 \rangle = \langle (s - f)_{\delta_e(s - f)} , (1)_1\rangle = \langle (f)_{\delta(f)}, (1)_1 \rangle$. Hence $\gr \profinggg{\tilde \delta} \cong \adelta[s]/\langle (f)_{\delta(f)}, (1)_1 \rangle = (\adelta/\langle (f)_{\delta(f)}, (1)_1 \rangle)[s]$. But $\adelta/\langle (f)_{\delta(f)}, (1)_1 \rangle$ $\cong$ $(\adelta/\langle (1)_1 \rangle)/ (\langle (f)_{\delta(f)}, (1)_1 \rangle/\langle (1)_1 \rangle)$ and $\langle (f)_{\delta(f)}, (1)_1 \rangle/\langle (1)_1 \rangle$ is precisely the ideal generated by $[(f)_{\delta(f)}]$ in $\gr \adelta$, which is $\gr \aaa$, while $\adelta/\langle (1)_1 \rangle \cong \gr \adelta$. It follows that $\gr \profinggg{\tilde \delta}  \cong (\gr \adelta/\gr \aaa)[s]$, and completes the proof of assertion \ref{graded-isomorphism}.

It remains only to prove assertion \ref{can-iterate-semidegree}. Due to assertion \ref{graded-isomorphism}, $\gr \aaa$ is a prime ideal of $\gr \adelta$ iff $\gr \profinggg{\tilde \delta}$ is a domain and, of course, iff $\langle (1)_1 \rangle$ is a prime ideal of $\profinggg{\tilde \delta}$. But $\langle (1)_1 \rangle$ is a prime ideal of $\profinggg{\tilde \delta}$ iff $\tilde \delta$ is a semidegree (theorem \ref{gengfsgf-characterization}), which completes the proof of the theorem.
\end{proof}

Theorem \ref{iterated-thm} motivates the following:
\begin{defn*}
Let $\delta$ be a semidegree on $A$. The {\em leading form} $\ld_\delta(f)$ of an element $f$ of $A$ is the equivalence class $[(f)_{\delta(f)}]$ of $(f)_{\delta(f)}$ in $\gr\adelta$.   
\end{defn*}

If $\delta$ is a semidegree on $A$ and the ideal $\langle \ld_\delta(f)\rangle$ of $\gr \adelta$ generated by the leading form $\ld_\delta(f)$ of $f \in A$ is prime, then $\tilde \delta$ is also a semidegree on $A$ (theorem \ref{iterated-thm}). Semidegree $\tilde \delta$ differs from $\delta$ according to assertion \ref{trivial-assertions} of theorem \ref{iterated-thm}. On the other hand, assertion \ref{iterated-equality} of theorem \ref{iterated-thm} shows that $\tilde \delta$ agrees with $\delta$ off $\gr \aaa$. We will say that $\tilde \delta$ is formed by the {\em iteration procedure} starting with semidegree $\delta$ by means of $f \in A$.

\begin{example}\label{iterated-semi-baby-example-p2}
Let $A := \kk[x_1,x_2]$ and $\delta$ be the semidegree on $A$ defined in \ref{iterated-semi-baby-example}. Recall that $\delta(x_1) = 3$, $\delta(x_2) = 2$ and $\delta(x_1^2-x_2^3)= 1$. Moreover, $\kk$-algebra $\adelta$ coincides with $\kk[(1)_1, (x_1)_3, (x_2)_2, (x_1^2-x_2^3)_1] = \kk[X_1,X_2,Y,Z]/\langle YZ^5 - X_1^2 +X_2^3 \rangle$. We claim that $\delta$ is formed by an iteration procedure by means of $f:= x_1^2-x_2^3$ starting with the weighted degree $\eta$ which assigns weight $3$ to $x_1$ and $2$ to $x_2$. Indeed, $\gr \profinggg{\eta} \cong \kk[x_1, x_2]$ via the map that sends $\ld_\eta(h) \in \gr \profinggg{\eta}$ to the leading weighted homogeneous component of $h$. Then, since $f = x_1^2-x_2^3$ is weighted homogeneous, $\ld_\eta(f) = f$. Since $\ld_\eta(f) = f \in  \kk[x_1,x_2] = \gr \profinggg{\eta}$, it follows that ideal $\langle \ld_\eta(f) \rangle$ is prime. Therefore according to assertion \ref{semi-iteration} of theorem \ref{iterated-thm}, degree like function $\tilde \eta$ formed by the iteration procedure by means of $f$ starting with $\eta$ is in fact a semidegree. Also $\profinggg{\tilde \eta} = \profingg{A[s]}{\eta_e}/\langle (s - f)_6 \rangle$, where $\eta_e$ is the weighted degree on $A[s]$ that extends $\eta$ and sends $s$ to $1$, as defined in the paragraph preceding lemma \ref{semi-extension}. Then with $t := (1)_1$, $\profingg{A[s]}{\eta_e}/\langle (s - x_1^2 + x_2^3)_6 = \kk[x_1,x_2,s,t]/\langle st^5 - x_1^2 + x_2^3 \rangle \cong \adelta$. To summarize, semidegree $\delta$ of example \ref{iterated-semi-baby-example} coincides with the iterated semidegree $\tilde \eta$.
\end{example}

\begin{rem}
Assume a semidegree $\delta$ on the coordinate ring $A$ of an affine variety $X$ is constructed by means of finitely many iterations starting with a semidegree $\eta$. Denote by $X^\eta$ and $\xdelta$ the completions of the $d$-uple embedding of $X$ into appropriate projective spaces (valid for appropriate $d \in \zz_+$ \cite[Lemma in section III.8]{mumred}). Then we can express the degree of $\xdelta$ in terms of the degree of $X^\eta$ (in a straightforward generalization of theorem \ref{iterated-degree-thm} below). In particular, in the special case of $\eta$ being a weighted homogeneous degree on $A  := \kk[x_1, \ldots, x_n]$ with weights $0 < d_i := \eta(x_i)$, $1 \leq i \leq n$, $\deg X^\eta = \frac{1}{d_1\cdot \cdots \cdot d_n}$ (example \ref{example-weighted-bezout}) and an explicit formula for $D := \deg \xdelta$, which appears in the affine Bezout-type theorem \ref{thm-affine-bezout}, follows:
\end{rem}

\begin{thm}
Let $\delta$ be a complete degree like function on the coordinate ring $A$ of an affine variety $X$, $f \in A$ and $w \in \zz$ with $0< w < \delta(f)$. Let $\delta_e$ and $\tilde \delta$ be degree like functions respectively on $A[s]$ and $A$ defined as above. Finally, let $d \in \zz_+$ be such that both $\xdelta$ and $X^{\tilde \delta}$ embeds into a usual projective space $\pp^l(\kk)$ via the $d$-uple embedding, and $D$ (resp. $\tilde D$) be the degree of the image of $\xdelta$ (resp. $\xtildedelta$) in $\pp^l(\kk)$. If the ideal $\scrI := \langle (s - f)_{\delta_e(s - f)} \rangle$ of $A[s]^{\delta_e}$ is prime, then $\tilde D = \frac{e}{w}D$.
\end{thm}

\begin{proof}
1. $\atildedelta \cong \profingg{A[s]}{\delta_e} / \scrI$.

2. The homomorphism defined by $\adelta[s] \ni \sum (f_i)_{d-iw}s^i \mapsto (\sum f_is^i)_d \in \profingg{A[s]}{\delta_e}$ is an isomorphism.

3. Let $e := \delta(f)$ and $t := (1)_1 \in \adelta$. Then $\scrI = \langle (f)_e - st^{e-w} \rangle$.

4. Let $(1)_1, (f_1)_{d_1}, \ldots, (f_k)_{d_k}$ generate $\adelta$ as a $\kk$-algebra. Then there is a surjection $\Phi: \kk[T, Y_1, \ldots, Y_k] \onto \adelta$, which induces a surjection $\Phi_e: \kk[T, Y_1, \ldots, Y_k, S] \onto \adelta[s] \cong \profingg{A[s]}{\delta_e}$. Let $\scrJ := \ker \Phi$, and $F$ be a weighted homogeneous polynomial in $Y_j$'s such that $\bar \delta(F) = e$ and $\Phi(F) = (f)_e$. Then $\ker \Phi_e = \langle \scrJ, F - ST^{e-w} \rangle$. 

5. The surjections in the previous steps induces embeddings of the form: 
$$\xymatrix{
\xtildedelta \ar@{^{(}->}[0,1] \ar@{-->}[1,0]& \WP' := \pp^{k+1}(\kk;1, d_1, \ldots, d_k, w) \ar@{<-_{)}}[1,0] \\
\xdelta \ar@{^{(}->}[0,1] 					& \WP := \pp^k(\kk;1, d_1, \ldots, d_k)} $$
Choose $d$ such that the $d$-uple embedding $\psi_d$ embeds $\pp^{k+1}(\kk;1, d_1, \ldots, d_k, w)$ into a usual projective space $\pp^l(\kk)$. 

6. Let $Y := V(\scrJ) \subseteq \WP'$. Then $\deg \psi_d(Y)$ is the number of intersections of $Y$ with $n+1$ generic hypersurfaces of weighted degree $d$, which equals $\frac{d}{w}D$.

7. Since $I(\xtildedelta)= I(Y) + \langle F - ST^{e-w} \rangle$, $\deg \psi_d(Y)$ is the number of intersections of $Y$ with $n$ generic hypersurfaces of weighted degree $d$ and $V(F - ST^{e-w})$ which equals $\frac{d}{w}D \times \frac{e}{d} = \frac{e}{w}D$.
\end{proof}

\begin{cor}[see {\cite[Example 9]{announcement}} and {\cite[Theorem 3.1.5]{preprint}}] \label{iterated-degree-thm}
Let $\delta_0$ be a weighted degree on $A := $ $\kk[x_1, \ldots, x_n]$. Let $k \geq 1$ and for each $i = 1, \ldots, k$, let $\delta_i$ be a semidegree on $A$ obtained by an iteration procedure starting with $\delta_{i-1}$ by means of a polynomial $h_i$ $($ with $\langle \ld_{\delta_{i-1}} (h_i) \rangle$ being prime in $\gr \profinggg{\delta_{i-1}})$ by assigning to the polynomial $h_i$ a weight $w_i$ with $0 < w_i < \delta_{i-1}(h_i)$. Then
\begin{align} \label{iterated-D}
\frac{D}{d^n} = \frac{1}{\delta_0(x_1)\cdots \delta_0(x_n)}\frac{\delta_0(h_1)}{w_1} \cdots \frac{\delta_{k-1}(h_k)}{w_k}\ , \tag{B}
\end{align}
where $D := \deg \xdelta$ and $d \in \zz_+$ are as in theorem \ref{thm-affine-bezout} for $X := \ank$, $A := \kk[x_1, \ldots, x_n]$ and $\delta := \delta_k$.
\end{cor}


\begin{proof}
Let $e_i := \delta_{i-1}(h_i)$ for $1 \leq i \leq k$. According to assertion \ref{semi-iteration} of theorem \ref{iterated-thm}, rings $A^{\delta_i}$ $\cong$ $(A^{\delta_{i-1}}[s_i])^{\delta^e_{i-1}}/\langle (s_i - h_i)_{e_i} \rangle$, where $s_i$ are indeterminates and $\delta^e_{i-1}$ extend $\delta_{i-1}$ by assigning weights $w_i$ to $s_i$. It follows by induction on $i$ with $x_0 = (1)_1$ that
$$A^{\delta_i} = \kk[x_0, \ldots, x_n, s_1, \ldots, s_i]/J_i\ ,$$
where $J_i := \langle \tilde h_1 - x_0^{e_1-w_1}s_1, \ldots, \tilde h_i - x_0^{e_i-w_i}s_i \rangle$, $1 \leq i \leq k$, and $\tilde h_j \in \kk[\bar x, \bar s]$ are weighted homogeneous polynomials in $(\bar x, \bar s) := (x_0, \ldots, x_n, s_1, \ldots, s_{j-1})$ whose equivalence classes in $\profinggg{\delta_i}$ are $(h_j)_{e_j}$, $1 \leq j \leq i$. 

Let $\tilde\delta$ be the weighted degree on $R_k := \kk[x_0, \ldots, x_n, s_1, \ldots, s_k]$ which assigns weight $1$ to $x_0$, $d_i := \delta_0(x_i)$ to $x_i$, $1\leq i \leq n$, and $w_j$ to indeterminates $s_j$, $1 \leq j \leq k$. Then homomorphism $\pi :R_k/J_k \to \profinggg{\delta}$ of graded $\kk$-algebras is surjective, and, therefore, for each $f \in A$, there is a polynomial $\tilde f$ in $R_k$ with $\tilde \delta(\tilde f) = \delta(f)$ and $\tilde f \mapsto (f)_{\delta(f)}$ under homomorphism $\pi$. Moreover, homomorphism $\pi$ induces an embedding of $X^{\delta_k}$ into the weighted projective space $\WP := \pp^{n+k}(\kk;1,d_1,\ldots, d_n, w_1, \ldots, w_k)$. Since $J_k$ is generated by exactly $k$ polynomials in $R_k$, it follows that the image of $\xdelta$ in $\WP$ is a {\em complete intersection}. Identifying $\xdelta$ with its image in $\WP$, it follows that $\xdelta = V(J_k)$ and, therefore, that for any $f_1, \ldots, f_n \in \kk[x_1, \ldots, x_n]$
\begin{gather}\label{generic-preservation}
\bigcap_{i=1}^n \{x \in \ank: f_i(x) = 0\} \subseteq \xdelta \cap V(\tilde f_1) \cap \cdots \cap V(\tilde f_n) =\\
										  V(\tilde h_1-x_0^{e_1-w_1}s_1) \cap \cdots \cap V(\tilde h_k - x_0^{e_k-w_k}s_k) \cap V(\tilde f_1) \cap \cdots \cap V(\tilde f_n).\notag
\end{gather}

Arguing via an embedding of $\WP \into \pp^N(\kk)$ it suffices to choose as $f_j$'s the pull backs of generic linear polynomials on $\pp^N(\kk)$ and then the intersection on the right hand side of the equality in \eqref{generic-preservation} would consist of isolated points in $X := \ank \into \xdelta \into \pp^N(\kk)$ (of multiplicities one and of the total number being the degree of $\xdelta$ in $\pp^N(\kk)$ according to the commonly used geometric definition of degree of a projective variety \cite[Definition 18.1]{harriseometry}). Consequently due to weighted homogeneous Bezout theorem (example \ref{example-weighted-bezout}) on $\affine{n+k}{\kk}$, the sum of the intersection multiplicitiess of Cartier divisors corresponding to $\tilde h_i - x_0^{e_i-w_i}s_i$, $1 \leq i \leq k$, and $\tilde f_j$, $1 \leq j \leq n$, at the points in the left hand side of the equality in \eqref{generic-preservation} is $\frac{\delta(f_1) \cdots \delta(f_n)e_1\cdots e_k}{d_1\cdots d_n w_1\cdots w_k}$, and by the Bezout theorem for semidegrees, the sum of the multiplicities of the fiber $\finv(0)$ at the points in the left hand side of the inclusion in \eqref{generic-preservation} is $\frac{D}{d^n}\delta(f_1)\cdots \delta(f_n)$. Formula \eqref{iterated-D} follows by comparing these two expressions, which completes the proof.
\end{proof}


\begin{example}
Let $f_k := (x_1 + (x_1^2 - x_2^3)^2, (x_1^2 - x_2^3)^k):\affine{2}{\kk} \to \affine{2}{\kk}$. We estimate the size of fibers of $f_k =: (f_{k1}, f_{k2})$ in three different ways. The first one is by means of the weighted homogeneous Bezout formula \eqref{weighted-bezout}. It is straightforward to see that the smallest upper bound given by \eqref{weighted-bezout} for $f_k$ is achieved for $d_1 = 3p$ and $d_2 = 2p$ for some $p \geq 1$, in which case the bound is $\frac{12p \cdot 6kp}{3p \cdot 2p} = 12k$. The second approach we take is via Bernstein's theorem (see section \ref{subsec-background}).  Let $a := (a_1, a_2) \in \affine{2}{\kk}$ and let $\scrP$ and $\scrQ_k$ be the Newton polygons of $x_1 + (x_1^2 - x_2^3)^2 - a_1$ and, respectively, of $(x_1^2 - x_2^3)^k - a_2$. The BKK bound for $|f_k^{-1}(a)|$ for non-zero $a_1, a_2$ is then $2\scrM(\scrP, \scrQ) = \vol(\scrP + \scrQ) - \vol(\scrP) - \vol(\scrQ) =  \frac{1}{2}(2k + 4)(3k+6) - 12 - 3k^2 = 12k$. (A similar argument implies that the BKK bound for $a_1$ or $a_2$ being zero is as well $12k$.)

Let $\delta$ be the iterated semidegree on $\kk[x_1,x_2]$ from example \ref{iterated-semi-baby-example-p2}, so that $\delta(x_1) = 3$, $\delta(x_2) = 2$ and $\delta(x_1^2-x_2^3) = 1$. Then $D/d^n = \frac{6}{3\cdot 2 \cdot 1} = 1$ (theorem \ref{iterated-degree-thm}). The estimate of $|f_k^{-1}(a)|$ given by \eqref{semi-bezout} with this $\delta$ is then $\delta(x_1 + (x_1^2 - x_2^3)^2)\delta((x_1^2 - x_2^3)^k) = 3k$. Moreover, ideals generated by $(f_{k1})_{\delta(f_{k1})}$, $(f_{k2})_{\delta(f_{k2})}$ and $(1)_1$ in $\adelta$ are primary to the irrelevant ideal of $\adelta$ and, therefore, completion $\psi_\delta$ preserves $f_k$ at $\infty$ over all points in $\affine{2}{\kk}$. In other words our bound with iterated semidegree $\delta$ is exact!

\end{example}
\section{Dimension 2 Revisited} \label{sec-dim-2}

\newcommand\divinfinitydeltad[2]{%
	\divinfinitydeltaa{#1}{d_{#2}}%
}

\newcommand{\clambda}{C^{\lambda}}
\newcommand{\glambda}{G^{\lambda}}
\newcommand{\xdelbracket}[1]{X^{\delta^{(m)}}}
  
In this section we continue the exploration of the relation of the number of solutions of a system of polynomials with subdegrees that preserve the system when $\dim X = 2$. At first we settle in this case question \ref{question-quasi-equality} of section \ref{subsec-bezout-subdegree} in the affirmative.

\begin{lemma} \label{semi-ideal-comparison}
If $\delta$ is a degree like function and $\eta$ is a semidegree such that $\delta \geq \eta$, then the ideal $\ppp_{\delta, \eta}$ of $\adelta$ generated by $\{(f)_d : d > \eta(f)\}$ is a prime ideal of $\adelta$. Moreover, if $\delta$ is a subdegree and $\eta$ is not an associated semidegree of $\delta$, then $V(\ppp_{\delta, \eta})$ has codimension at least $2$ in $\xdelta$.
\end{lemma}

\begin{proof}
Let $L := \{(f)_d : d > \eta(f)\}$. Since $\eta \leq \delta$, it follows that if $(f)_d \in L$ and $(g)_e \in \adelta$, then $(f)_d(g)_e = (fg)_{d+e} \in L$. This implies that $L$ is precisely the set of homogeneous elements of $\ppp_{\delta, \eta}$. Therefore, if $(f_1)_{d_1}, (f_2)_{d_2} \in \adelta \setminus \ppp_{\delta, \eta}$, then $\eta(f_i) = d_i$ for each $i$, $1 \leq i \leq 2$. It follows that $\eta(f_1f_2) = d_1 + d_2$ and hence $(f_1f_2)_{d_1d_2} \in \adelta \setminus \ppp_{\delta, \eta}$. Therefore $\ppp_{\delta, \eta}$ is prime. 

Now assume $\delta$ is a subdegree with minimal presentation $\delta = \max\{\delta_j: 1 \leq j \leq N\}$. Let $\ppp_j$ be the prime ideal of $\adelta$ corresponding to $\delta_j$. Since $(1)_1 \in \ppp_{\delta, \eta}$, it follows that $\ppp_{\delta, \eta} \supseteq \ppp_j$ for some $j$. 

\begin{claim}\label{semi-prime-comparison}
For each $j$, $1 \leq j \leq N$, $\ppp_{\delta, \eta} = \ppp_j$ iff $\eta = \delta_j$.
\end{claim}

\begin{proof}
The `if' direction follows directly from assertion \ref{part1-gensgf-characterize-lemma} of lemma \ref{gensgf-characterize-lemma}. We now show the `only if' direction. Assume $\ppp_{\delta, \eta} = \ppp_j$. There exists $(f)_d \in \adelta \setminus \ppp_j$ such that $\delta_j(f) > \delta_i(f)$ for all $i \neq j$. Let $g \in A$. Then there exists $k \in \nn$ such that $\delta_j(f^kg) = \delta(f^kg)$. Since $\ppp_{\delta, \eta} = \ppp_j$, it follows that $\eta(f) = \delta(f)$ and $\eta(f^kg) = \delta(f^kg)$. Therefore $\eta(g) = \eta(f^kg) - \eta(f^k) = \delta(f^kg) - \delta(f^k) = \delta_j(f^kg) - \delta_j(f^k) = \delta_j(g)$.
\end{proof}

By the above claim it follows that $\ppp_{\delta, \eta} \supsetneq \ppp_j$ for all $j$, which completes the proof of the lemma.
\end{proof}

{\bf Do you need $X$ to be normal to ensure that $[\divclosuredelta[g]_X(g)]$ is an effective CARTIER divisor??} I do not think so - look at the normalization $\overline {\xdeltaa{g}}$ of $X$. If $[\divclosuredelta[g]_X(g)]$ is not effective as a Cartier divisor on $\overline{\xdeltaa{g}}$ (i.e. at some point its local equation is not regular), then it necessarily has at least one pole $Y$ and it has to be contained in $\overline {\xdeltaa{g}}\setminus X$. Then the image of $Y$ is one of the components $V_j$ of infinity in $\xdeltaa{g}$. It follows that $\sheaf_{Y,\overline {\xdeltaa{g}}} \supseteq \sheaf_{V_j,\xdeltaa{g}}$ and hence by maximality of discrete valuation rings $\sheaf_{Y,\overline {\xdeltaa{g}}} = \sheaf_{V_j,\xdeltaa{g}}$. It follows that order of vanishing of the local equation of $[\divclosuredelta[g]_X(g)]$ along $Y$ is $0$ - a contradiction. 

\begin{prop} \label{delta-rational-map}
Let $\delta = \max\{\delta_j: 1 \leq j \leq N\}$ and $g \in A$ such that $\delta_j(g) > 0$ for all $j$ and both $\delta$ and $\delta_g$ are finitely generated. Pick positive integers $r \leq N$ and $m_1, \ldots, m_r$ such that $\delta'$ is a finitely generated subdegree with minimal presentation $\delta' := \max\{m_j\delta_j:1 \leq j \leq r\}$. Let $\phi: X^{\delta'} \to \xdeltaa{g}$ be the birational map induced by identification of $X$ and let $S \subseteq X^{\delta'}$ be the set of points of indeterminacy of $\phi$. Then for all $j$, $1 \leq j \leq r$, $\phi(V'_j \setminus S) \subseteq V(\ppp_{\delta_g, \tilde\delta_j})$ where $V'_j$ is the component of the hypersurface at infinity of $X^{\delta'}$ corresponding to $m_j\delta_j$, $\tilde\delta_j := \frac{e_g}{\delta_j(g)}\delta_j$ and $\ppp_{\delta_g, \tilde\delta_j}$ is as defined in lemma \ref{semi-ideal-comparison}.
\end{prop}

\begin{proof}
Assume contrary to the proposition that there exists $j$, $1 \leq j \leq r$, such that $\phi(V'_j \setminus S) \subsetneq V(\ppp_{\delta_g, \tilde\delta_j})$. Pick $x \in \phi(V'_j \setminus S) \setminus V(\ppp_{\delta_g, \tilde\delta_j})$. Let $(f)_d \in \ppp_{\delta_g, \tilde\delta_j}$ such that $x \not\in V((f)_d)$. 

Note that $x \in \xdeltaa{g}\setminus X = V((1)_1)$ and therefore for a suitable positive integer $k$, we may assume that the local equation of the $kd$-uple divisor $\divinfinitydeltaa{g}{kd}$ at infinity of $\xdeltaa{g}$ on a neighborhood $U$ of $x$ is $\frac{1}{f^k}$. According to assertion \ref{div-g-d-delta-g-d} of lemma \ref{divisor-at-infinity}, $kd[\divclosuredelta[g]_X(g)] = [\divdelta[g](g^{kd})] + e_g[\divinfinitydeltaa{g}{kd}]$ is a Cartier divisor on $\xdeltaa{g}$. By construction the local equation of $kd[\divclosuredelta[g]_X(g)]$ on $U$ is $\frac{g^{kd}}{f^{ke_g}}$. Note that
\begin{align*}
\delta_j(\frac{g^{kd}}{f^{ke_g}}) 	&= kd\delta_j(g) - ke_g\delta_j(f) \\
									&= k\delta_j(g)(d - \frac{e_g\delta_j(f)}{\delta_j(g)}) \\
									&= k\delta_j(g)(d - \tilde\delta_j(f))
\end{align*}
By assumption on $\delta$, $\delta_j(g) > 0$. Moreover $d > \tilde\delta_j(f)$, since $(f)_d \in \ppp_{\delta_g, \tilde\delta_j}$. It follows that $\delta_j(\frac{g^{kd}}{f^{ke_g}}) > 0$ and hence $\frac{g^{kd}}{f^{ke_g}}$ has a pole at $V'_j$, so that the pullback of $kd[\divclosuredelta[g]_X(g)]$ to $X^{\delta'}\setminus S$ is not effective. But this is impossible, since $kd[\divclosuredelta[g]_X(g)]$ is an effective Cartier divisor. This contradiction proves the proposition.
\end{proof}

\begin{prop} \label{pullback-cartier-function-closure}
Let $X$, $\delta$ and $g$ be as in proposition \ref{delta-rational-map}. Moreover assume $\dim X = 2$. Let $\lambda^i_j$, $1 \leq i \leq k$, $1 \leq j \leq N$, be positive integers and $X^\eta$ be the closure of the diagonal embedding of $X$ into $\xdeltaa{g} \times X^{\delta^1} \times \cdots \times X^{\delta^k}$, where $\delta^i := \max\{\lambda^i_j\delta_j: 1 \leq j \leq N\}$ for $1 \leq i \leq k$. Then $\pi^*(\divclosuredelta[g]_X(g)) = \divclosure[\eta]_X(g)$, where $\pi$ is the projection in the first coordinate.
\end{prop}

\begin{proof}
Let ${\overline{V(g)}}^{\delta_g} \cap \xdeltaa{g}\setminus X = \{x_1, \ldots, x_r\}$, where ${\overline{V(g)}}^{\delta_g}$ is the closure in $\xdeltaa{g}$ of $V(g) \subseteq X$. Now, $[\pi^*(\divclosuredelta[g]_X(g))] = [\divclosure[\eta]_X(g)] + E$, where $\supp E \subseteq X^\eta \setminus X$. Assume $E \not= \emptyset$. Then $E$ is of the form $\sum_{j=1}^s m_j[V_j]$ such that for each $j$, $1 \leq j \leq s$, $\pi(V_j) = x_{i_j}$ for some $i_j$, $1 \leq i_j \leq r$. WLOG we may assume $ i_1 = 1$. Since $\dim(V_1) = 1$, there exists $j$, $1 \leq j \leq k$, such that $\dim(\pi_j(V_1))$ is also $1$, where $\pi_j: X^\eta \to X^{\delta^j}$ is the natural projection map. WLOG assume $j=1$ and let $\phi: X^{\delta^1} \to \xdeltaa{g}$ be the birational map induced by the identification of $X$ in both spaces and $S \subseteq X^{\delta^1}$ be the set of points of indeterminacy of $\phi$. Since $\pi \equiv \phi \circ \pi_1$ on $X^\eta \setminus \pi_1^{-1}(S)$, it follows that $\phi(\pi_1(V_1)\setminus S) = \{x_1\}$.

WLOG we may order $\delta_1, \ldots, \delta_N$ in a way that $\delta^1$ has minimal presentation $\delta^1 = \max\{\lambda^1_j\delta_j: 1 \leq j \leq M_1\}$ for some $M_1 \leq N$ and $\pi_1(V_1)$ is the component of the hypersurface at infinity of $X^{\delta'}$ corresponding to $\lambda^1_1\delta_1$. Henceforth we write $V'_1$ for $\pi_1(V_1)$. Let $\tilde\delta_j := \frac{e_g}{\delta_j(g)}\delta$ for $1 \leq j \leq N$. Then $\delta_g = \max\{\tilde\delta_1, \ldots, \tilde\delta_N\}$. 

\begin{claim*}
$\tilde\delta_1$ is not an associated semidegree of $\delta_g$.
\end{claim*}

\begin{proof}
Assume $\tilde\delta_1$ is an associated semidegree of $\delta_g$. Then according to claim \ref{semi-prime-comparison} $\ppp_{\delta_g, \tilde\delta_1} = \tilde \ppp_1$, where $\tilde \ppp_1$ is the prime ideal of $\adeltaa{g}$ corresponding to $\tilde\delta_1$. Proposition \ref{delta-rational-map} then implies that $\phi(V'_1 \setminus S) \subseteq \tilde V_1 := V(\tilde \ppp_1)$. Since $\tilde \delta_1$ and $\lambda^1_1\delta_1$ induces the same discrete valuation $\nu$ on $\kk(X)$, it follows due to proposition \ref{pole-and-degree} that $\sheaf_{V'_1, X^{\delta^1}} = \sheaf_{\tilde V_1, \xdeltaa{g}}$, both of these rings being same as the valuation ring of $\nu$. It follows that $\overline{\phi(V'_1 \setminus S)} = \tilde V_1$, which is absurd, since $\phi(V'_1 \setminus S) = \{x_1\}$. This contradiction proves the claim.
\end{proof}

Since $\tilde\delta_1$ is not an associated semidegree of $\delta_g$, lemma \ref{semi-ideal-comparison} implies that $V(\ppp_{\delta_g, \tilde \delta_1})$ has codimension $2$ in $\xdeltaa{g}$. Since $V(\ppp_{\delta_g, \tilde \delta_1})$ is also irreducible, it follows that $V(\ppp_{\delta_g, \tilde \delta_1})$ is a single point. Let $x$ be the sole element of $V(\ppp_{\delta_g, \tilde \delta_1})$. 

\begin{claim*}
$x \not\in {\overline{V(g)}}^{\delta_g}$.
\end{claim*}

\begin{proof}
Since $\tilde\delta_j(g) = e_g$ for all $j$, $1 \leq j \leq N$, it follows that ${\overline{V(g)}}^{\delta_g} = V((g)_{e_g})$ (remark \ref{sub-all-equal-degree}). Moreover, $\tilde\delta_1(g) = e_g = \delta_g(g)$, so that $(g)_{e_g} \not \in \ppp_{\delta_g, \tilde \delta_1}$. Therefore $x \not\in V((g)_{e_g}) = {\overline{V(g)}}^{\delta_g}$, as required.
\end{proof}

According to proposition \ref{delta-rational-map} $\{x_1\} = \phi(V'_1\setminus S) \subseteq V(\ppp_{\delta_g, \delta_1}) = \{x\}$. But this contradicts the above claim. It follows that $E = \emptyset$ and hence $[\pi^*(\divclosuredelta[g]_X(g))] = [\divclosure[\eta]_X(g)]$, as required.
\end{proof}

\begin{cor}\label{quasi-affine-two-bezout-equality}
Assume $\dim X = 2$ and that $f := (f_1, f_2):X \to \affine{2}{\kk}$, $\delta := \max\{\delta_j: 1 \leq j \leq N\}$ and $d_{f_1}, d_{f_2} \in \nn$ are as in theorem \ref{quasi-affine-bezout-prelim}. If $\xdelta$ preserves $\{f_1, f_2\}$ at $\infty$ then for almost all $a \in X$, \eqref{quasi-bezout} holds with an equality.
\end{cor}

\begin{proof}
Note that for each $a \in X$, replacing $f$ by $f - a$ does not affect the assumptions on $f$. Therefore it suffices to show that if $\xdelta$ preserves $\{f_1, f_2\}$ at $\infty$ over $0$ then \eqref{quasi-bezout} holds with an equality for $a = 0$.

Let $X^{\eta'}$ (resp. $X^\eta$) be the closure of the diagonal embedding of $X$ into $X^\delta \times X^{\delta_{f_1}} \times X^{\delta_{f_2}}$ (resp. $X^{\delta_{f_1}} \times X^{\delta_{f_2}}$). Then there is a system of maps as follows:
$$\xymatrix @R-1pc@C-1pc{
								& 													& X^{\eta'}	\ar[1,-1]_{\pi'} \ar[2,2]^{\pi'_0}	& &	\\	
								& X^\eta \ar[1,-1]_{\pi_1} \ar[1,1]^{\pi_2}	&																& &	\\
X^{\delta_{f_1}}&											& X^{\delta_{f_2}}							& & X^\delta
}$$
such that each map is the identity on $X$. Fix an $i$, $1 \leq i \leq 2$. Let $D_i := \pi_i^*(\divinfinitydelta{f_i})$ and $D'_i := {\pi'}^*(D_i) = {\pi'_i}^*(\divinfinitydelta{f_i})$, where $\pi'_i := \pi_i \circ \pi'$. Due to lemma \ref{divisor-at-infinity}, on $X^{\eta'}$, $[\Div_{X^{\eta'}}(f_i^{d_{f_i}})] = d_{f_i}[{\pi'_i}^*(\divclosuredelta[f_i]_X(f_i))] - e_{f_i}[D'_i]$. Since ${\pi'_i}^*(\divclosuredelta[f_i]_X(f_i)) = \divclosure[\eta']_X(f_i)$ according to proposition \ref{pullback-cartier-function-closure}, it follows that $[\Div_{X^{\eta'}}(f_i^{d_{f_i}})] = d_{f_i}[\divclosure[\eta']_X(f_i)] - e_{f_i}[D'_i]$ and therefore 
$$(D'_1, D'_2) = \frac{d_{f_1}d_{f_2}}{e_{f_1}e_{f_2}} (\divclosure[\eta']_X(f_1), \divclosure[\eta']_X(f_2))~.$$
By assumption $\xdelta$ preserves $\{f_1, f_2\}$ at $\infty$. It follows that $X^{\eta'}$ also preserves $\{f_1, f_2\}$ at $\infty$. Therefore the points at the intersection of supports of $\divclosure[\eta']_X(f_1)$ and $\divclosure[\eta']_X(f_2)$ are precisely the points of $\finv(0)$ a intersections of $V(f_1)$ and $V(f_2)$. This implies that 
\begin{align}
|\finv(0)| = (\divclosure[\eta']_X(f_1), \divclosure[\eta']_X(f_2)) = \frac{e_{f_1}e_{f_2}}{d_{f_1}d_{f_2}}(D'_1, D'_2)~. \label{solutions-equal-intersection}
\end{align}
Moreover intersection numbers are preserved under the pull backs by proper birational morphisms \cite[Example 2.4.3]{fultersection} and therefore $(D'_1, D'_2) = ({\pi'}^*(D_1), {\pi'}^*(D_2)) = (D_1, D_2)$. Recall that $(D_1, D_2) = \deg (s(X^\eta))$ according to \eqref{intersection-equals-degree}, where $s$ is the Segre embedding of $X^\eta$ into the product of ambient spaces of $\xdeltaa{f_1}$ and  $\xdeltaa{f_2}$. Combining the latter equality with \eqref{solutions-equal-intersection}, we obtain the desired equality.
\end{proof}

\begin{lemma} \label{map-existence}
Let $\eta_1, \ldots, \eta_k$ be complete degree like functions on $A$. Then there is a complete degree like function $\eta$ and proper maps $\phi_i: X^\eta \to X^{\eta_i}$ for $i = 1, \ldots, k$ such that the following diagram commutes for each $i$:
$$\xymatrix{
X^\eta \ar[r]^{\phi_i} 	& X^{\eta_i} \\
X \ar@{^{(}->}[u]|{\psi_\eta} \ar[r]^{\mathds{1}} & X \ar@{^{(}->}[u]|{\psi_{\eta_i}}
}\ .$$
\end{lemma}

\begin{proof}
Clear: let $X^\eta$ be the closure of the diagonal embedding of $X$ into $X^{\eta_1} \times \cdots \times X^{\eta_k}$.
\end{proof}

Assume for all $\lambda := (\lambda_1, \ldots, \lambda_N)$ with $\lambda_j \geq 1$, $\delta^{(\lambda)} := \max\{\lambda_j\delta_j: 1 \leq j \leq N\}$ is finitely generated. Let $\{\lambda^i: i \geq 1\}$ be an enumeration of $\nn^2$. For each $i$, let $g_i := (f_1)^{\lambda^i_1d_{f_1}} (f_2)^{\lambda^i_2d_{f_2}}$. As above, define quasidegree $\delta_{g_i}$ on $A$, choose a suitable $d_{g_i} \in \nn$ and let $D_{\delta_{g_i}, d_{g_i}}$ be the divisor at $\infty$ on $X^{\delta_{g_i}}$. By lemma \ref{map-existence}, there exist completions $X_1, X_2, \ldots $ of $X$ and a system of maps as follows:
$$\xymatrix @R-1pc@C-1pc{
							& X^{\delta_{g_i}}  & X^{\delta_{g_{i-1}}}	& 							& X^{\delta_{g_2}} 	& X^{\delta_{g_1}}									& \\
\cdots \ar[r] & X_i \ar[u] \ar[r] & X_{i-1} \ar[u] \ar[r] & \cdots \ar[r]	& X_2 \ar[u] \ar[r] & X_1 \ar[u]\ar[1,-1]\ar[1,1]\ar[r] & X^\delta \\
							&										&												&								&	X^{\delta_{f_1}}	&	\cdots														& X^{\delta_{f_n}}
}$$

Fix $K \geq 1$. For each $j \leq K$ and each $k \geq K$, we can pull back on $X_k$ the Cartier divisor $D_{\delta_{g_j}, d_{g_j}}$ and we can also pull back each $D_{\delta_{f_i}, d_{f_i}}$ for $i= 1, 2$. If $D_1, D_2$ are Cartier divisors on $X_k$ which are linear combinations of those in the preceeding sentence, the intersection product $(D_1, D_2)$ is independent of $k$ \cite[Example 2.4.3]{fultonsection}. In particular, $(D_{\delta_{g_j}, d_{g_j}})^2$ and $(D_{\delta_{f_{i_1}}, d_{f_{i_1}}}, D_{\delta_{f_{i_2}}, d_{f_{i_2}}})$ are well defined for each $j \leq 1$ and $1 \leq i_1, i_2 \leq n$. Now fix any $j$. Let $V_{g_j}$ (resp. $V_{f_i}$ for $1 \leq i \leq 2$) be the (principal) Cartier divisor on $X$ generated by $g_j$ (rep. $f_i$). Then $V_{g_j} = \sum_{i=1}^2 \lambda^j_id_{f_i}V_{f_i}$. Applied to $X_j$, lemma \ref{divisor-at-infinity} implies that $((f_i)^{d_{f_i}}) = d_{f_i}\bar V_{f_i} - D_{\delta_{f_i}, d_{f_i}}$ for each $i$, so that $(g_j) = \sum_{i=1}^2 \lambda^j_id_{f_i} (f_i) = \sum_{i=1}^2 \lambda^j_i(d_{f_i}\bar V_{f_i} - D_{\delta_{f_i}, d_{f_i}}) = \bar V_{g_j} - \sum_{i=1}^2 \lambda^j_i D_{\delta_{f_i}, d_{f_i}}$. By lemma \ref{divisor-at-infinity} again, it follows that $D_{\delta_{g_j}, d_{g_j}} = d_{g_j}\sum_{i=1}^2 \lambda^j_i D_{\delta_{f_i}, d_{f_i}}$. Therefore, the function $\scrM:\nn^2 \to \nn$ defined by:
\begin{align*}
\scrM(\lambda^j) := \frac{(D_{\delta_{g_j}, d_{g_j}})^2}{(d_{g_j})^2}
\end{align*}
depends polynomially on its arguments, since 
\begin{align} \label{M-expression}
\scrM(\lambda_1, \lambda_2) &
	= (\sum_{i=1}^2 \lambda_i D_{\delta_{f_i}, d_{f_i}})^2 
	=\lambda_1^2 D^2_{\delta_{f_1}, d_{f_1}} + 2 \lambda_1\lambda_2(D_{\delta_{f_1}, d_{f_1}}, D_{\delta_{f_2}, d_{f_2}}) + \lambda_2^2D_{\delta_{f_2}, d_{f_2}}^2.
\end{align}

We say that $\nu$ {\em separates} $\delta_1, \ldots, \delta_N$ if for all $\lambda_1, \ldots, \lambda_N \in \nn$, $\bigcap_{i=1}^N \nu(F^{\lambda_i\delta_i}_k) = \nu(\bigcap_{i=1}^N F^{\lambda_i\delta_i}_k)$ for all sufficiently large $k$. 

\begin{example}
Let $\prec$ be any total ordering on $\zz^n$ such that it is compatible with addition. Define $\nu: \cc[x_1, \ldots, x_n] \to \zz^n$ by: $\nu(\sum_\alpha a_\alpha x^\alpha) := \min_{\prec}\{\alpha: a_\alpha \neq 0\}$. Then $\nu$ separates all the weighted degrees in $x_1, \ldots, x_n$ coordinates.
\end{example}


\begin{example}
Let $w_1 := x - y^2$ and $w_2 := x+ y^2$. Then $\cc[x,y] = \cc[w_1,y] = \cc[w_2,y]$. Let $\delta_i$ be the degree in $(w_i,y)$ coordinates, $1 \leq i \leq 2$, and $\nu$ be any valuatoin on $\cc[x,y]$ induced by an ordering of monomials in $(x,y)$. Then $\prec$ does {\em not} separate $\delta_1$ and $\delta_2$. Indeed, $F^{\delta_1}_1 = \cc\langle 1, y, x - y^2 \rangle$ and $F^{\delta_2}_1 = \cc\langle 1, y, x + y^2 \rangle$. Consequently, $F^{\delta_1}_1 \cap F^{\delta_2}_1 = \cc\langle 1, y \rangle$ and $\nu\left(F^{\delta_1}_1 \cap F^{\delta_2}_1 \right) = \{(0,0), (0,1)\}$. On the other hand, $\nu\left(F^{\delta_1}_1\right) \cap \nu\left(F^{\delta_2}_1\right) = \{(0,0), (0,1), (1,0)\}$.
\end{example}

Now let $X$, $\delta$, $f$ be as in Theorem \ref{quasi-affine-bezout-prelim} and let $\nu$ be a valuation on $A := \kk[X]$ that {\em separates} $\delta_1, \ldots, \delta_N$. Let $C_j$ be the smallest closed cone in $\rr^3$ containing
$$G_j := \{(\frac{1}{d_{g_j}}\delta_{g_j}(h),\nu(h)) \in \zz_+^3: h \in A\}.$$

For each $\lambda \in \nn^n$, let $\clambda$ be the smallest cone in $\rr^3$ containing
$$\glambda := \{(\frac{1}{d_{g_j}}\delta_{g_j}(h),\nu(h)) \in \zz_+^3: h \in A\}.$$

By lemma \ref{divisor-at-infinity} and proposition \ref{propvanskii}, $(D_{\delta_{g_j}, d_{g_j}})^2 = 2\vol(\Delta_j)$, where $\Delta_j$ is the convex hull of the cross-section of $C_j$ at the first coordinate value $1$. Since $\nu$ separates $\delta_1, \ldots, \delta_N$, it follows that $C_j = C_{j,1} \cap \cdots \cap C_{j,N}$, where for each $k$, $C_{j,k}$ be the smallest closed cone in $\rr^3$ containing
$$G_{j,k} := \{(\frac{1}{d_{g_j}\delta_k(g_j)}\delta_k(h),\nu(h)) \in \zz_+^3: h \in A\}.$$
It follows that
$$\scrM(\lambda_1, \lambda_2) = \vol\left(\bigcap_{j=1}^N \left(\lambda_1 d_{1j} + \lambda_2 d_{2j}\right) \Delta^{(j)}\right) $$
For appropriate $d_{ij}$ and $\Delta^{(j)}$'s.\\

Now, \cite[Therorem 6.4]{lazarsfeld-mustata} implies that each $\Delta^{(j)}$ has {\em linear} edges. Since $\scrM$ is a homogeous polynomial of degree $2$ in $\lambda_i$'s, this forces that the sums and intersection commute in the preceding expression for $\scrM$, i.e.\ 
$$\scrM(\lambda_1, \lambda_2) = \vol\left(\lambda_1 \bigcap_{j=1}^Nd_{1j}\Delta^{(j)} +\lambda_2 \bigcap_{j=1}^N d_{2j} \Delta^{(j)}\right).$$
It follows from comparing the preceding expression with \eqref{M-expression}, that $(D_{\delta_{f_1}, d_{f_1}}, D_{\delta_{f_2}, d_{f_2}})$ is precisely the mixed volume of $\bigcap_{j=1}^Nd_{1j}\Delta^{(j)}$ and $ \bigcap_{j=2}^Nd_{2j}\Delta^{(j)}$, as required.

\begin{rem*}
The theorem has to be stated explicitly!
\end{rem*}

\bibliographystyle{plain}
\bibliography{bibi}


\end{document}